\documentclass[11pt]{article}
\usepackage[left=0.9in,right=0.9in,top=1in,bottom=1in]{geometry}
\usepackage{amssymb,amsmath,amsthm}
\usepackage{mathrsfs}
\usepackage[active]{srcltx}
\usepackage{verbatim}
\usepackage{graphicx}
\usepackage[numbers,sort]{natbib}

\usepackage{hyperref}
\usepackage{caption}
\usepackage{todonotes}
\usepackage{enumerate}
\usepackage{mathtools}
\usepackage{bbm}

\numberwithin{equation}{section}

\theoremstyle{definition}
 \newtheorem*{problem}{Open problem}

 \allowdisplaybreaks

\mathtoolsset{showonlyrefs,showmanualtags}

\newcommand\ud{\,\mathrm{d}}
\newcommand\ee{\mathrm{e}}

\DeclareMathOperator{\EE}{\mathbb{E}}
\DeclareMathOperator{\PP}{\mathbb{P}}

\newcommand{\Ee}{\mathcal{E}}
\newcommand{\Ff}{\mathcal{F}}

\newcommand{\Hh}{\mathcal{H}}

\newcommand{\DD}{\mathbb{D}}

\newcommand{\HH}{\mathbb{H}}

\newcommand{\NN}{\mathbb{N}}

\newcommand{\RR}{\mathbb{R}}

\newcommand{\ZZ}{\mathbb{Z}}

\newcommand{\HHP}{\mathbb{H}_\mathrm{P}}
\newcommand{\DDP}{\mathbb{D}_\mathrm{P}}
\newcommand{\DDK}{\mathbb{D}_\mathrm{K}}

\newcommand{\tW}{\widetilde{W}}
\newcommand{\oEe}{\overline{\Ee}}

\newcommand{\1}[1]{{\mathbbm{1}\mkern -1.5mu}{\{#1\}}}

\newcommand{\expo}[1]{{\mathrm{Exp} ( #1 )}}

\newcommand{\RP}{{\mathbb R}_+}
\newcommand{\unif}{\mathrm{Unif}\mkern2mu}

\DeclareMathOperator{\Var}{\mathbb{V}\mathrm{ar}}
\DeclareMathOperator{\arctanh}{arctanh}

\DeclareMathOperator{\conv}{conv}
\DeclareMathOperator{\perim}{perim}

\newcommand{\as}{\ \text{a.s.}}
\newcommand{\eqd}{\overset{{d}}{=}}

\theoremstyle{plain}
\newtheorem{theorem}{Theorem}[section]
\theoremstyle{remark}
\newtheorem{remark}[theorem]{Remark}
\newtheorem{example}[theorem]{Example}
\theoremstyle{plain}
\newtheorem{corollary}[theorem]{Corollary}
\newtheorem{lemma}[theorem]{Lemma}
\newtheorem{proposition}[theorem]{Proposition}

\numberwithin{equation}{section}

\allowdisplaybreaks

\title{Perimeter length of the convex hull of Brownian motion\\ in the hyperbolic plane}

\author{Chinmoy Bhattacharjee\footnote{\raggedright Department of Mathematics, University of Hamburg, Germany;~\href{mailto:chinmoy.bhattacharjee@uni-hamburg.de}{\texttt{chinmoy.bhattacharjee@uni-hamburg.de}}.} 
\and Rik Versendaal\footnote{\raggedright Delft Institute of Applied Mathematics, TU Delft,   Netherlands;~\href{mailto:r.versendaal@tudelft.nl}{\texttt{r.versendaal@tudelft.nl}}.} 
\and Andrew Wade\footnote{\raggedright Department of Mathematical Sciences, Durham University,  UK;~\href{mailto:andrew.wade@durham.ac.uk}{\texttt{andrew.wade@durham.ac.uk}}.}}

\date{\today}

\begin{document}

\maketitle

\begin{abstract}
We relate the expected hyperbolic length of the perimeter of the convex hull 
of the trajectory of Brownian motion in the hyperbolic plane 
to an expectation of a certain exponential functional of a one-dimensional real-valued Brownian motion, and hence derive small- and large-time asymptotics for the expected hyperbolic perimeter. In contrast to the case of Euclidean Brownian motion with non-zero drift, the large-time asymptotics are a factor of two greater than the lower bound implied by the fact that the convex hull includes the hyperbolic line segment from the origin to the endpoint of the hyperbolic Brownian motion. We also obtain an exact expression for the expected perimeter length after an independent exponential random time.
\end{abstract}

\medskip


\medskip


\section{Introduction}
\label{sec:intro}

\subsection{Hyperbolic Brownian motion}
\label{sec:hyp-BM}

The hyperbolic plane $\HH^2$ is
the unique complete, simply-connected two-dimensional Riemannian manifold with constant curvature~$-1$. A point $o \in \HH^2$ is distinguished as the origin, and the hyperbolic distance $d_\HH$ equips $\HH^2$ with a norm and a metric. 
For distinct points $x, y \neq o$, there are unique hyperbolic geodesics
between $o, x$ and between $o,y$, and these geodesics subtend an angle at $o$ which is zero only if $x$ lies on the geodesic from $o$ to $y$, or vice versa. 
Every point $x \in \HH^2 \setminus \{ o \}$ can be represented in geodesic polar coordinates
$ x = (r,\theta)$, where $r = d_\HH (o, x) \in (0,\infty)$, and $\theta \in [0,2\pi)$ is an angle relative to a fixed reference direction. 

\emph{Hyperbolic Brownian motion} is a stochastic process $B:=(B_t)_{t \in \RP}$
on $\HH^2$,  started from $B_0 := o$,
whose dynamics can be described in geodesic polar coordinates $B_t = (R_t, \theta_t) \in \RP \times [0,2\pi)$, $t > 0$.
The $\RP$-valued \emph{radial process} $R := (R_t)_{t \in\RP}$
starts from $R_0:=0$ and 
satisfies the autonomous stochastic differential equation (SDE),
in which  $W^R$ is a standard $\RR$-valued Brownian motion,
\begin{align}
\label{eq:r-sde}
    \ud R_t & = \frac{\ud t }{2 \tanh R_t} + \ud W_t^R , \text{ for all } t \in \RP. \end{align}
Since $B_0 = o$, the angle
 $\theta_0$ is not canonically defined, but for $t \geq s > 0$
 the increment $\theta_t - \theta_s$ is described via 
 an $\RR$-valued post-$s$ \emph{winding process} $\Theta^{(s)} := (\Theta^{(s)}_t)_{t \geq s}$
which starts from $\Theta^{(s)}_s :=0$ and satisfies the SDE
\begin{align}
    \label{eq:theta-sde}
    \ud \Theta^{(s)}_t & = \frac{\ud W^{(s)}_t}{\sinh R_t} , \text{ for all } t \geq s,
\end{align}
where $W^{(s)}$ is a standard Brownian motion, independent of~$W^R$. 
Then $\theta_t = ( \theta_s + \Theta^{(s)}_t )$ modulo $2\pi$. 
Combined with the entrance law $\theta_s \sim \unif [0,2\pi)$, for every $s >0$, inevitable thanks to the 
 \emph{rapid spinning} out from the origin~\cite[\S 7.16]{ito-mckean},
 the SDEs~\eqref{eq:r-sde}--\eqref{eq:theta-sde} provide a  unique-in-law description.  In particular, the law of $B$ is rotationally invariant.
This also gives a \emph{skew product} decomposition of Brownian motion on $\HH^2$, similarly to Euclidean Brownian motion in $\RR^d$, $d \geq 2$.
Note that from~\eqref{eq:r-sde} it can be shown that $R_t >0$ for all $t >0$, a.s.~(see Lemma~\ref{lem:0-hit} below) so that in~\eqref{eq:theta-sde}, $\sinh R_t > 0$ for $t>0$. (Started instead from $B_0 \neq o$, 
there is no issue with $\theta_0$ or rapid spinning, and~\eqref{eq:theta-sde} can be used to define the post-$0$ winding process $\Theta^{(0)}$ for all time.) 
 For background on Brownian motion on $\HH^2$, including the above facts, we refer to~\cite[\S V.36]{rw2} and~\cite[\S 2]{gruet1997}, for example.
 
The following result, part of which is well known (see Remark~\ref{rem:lln}),
 shows that hyperbolic Brownian motion $B$ exhibits ballistic transience and a random limiting direction. It quantifies the intuition, apparent from~\eqref{eq:r-sde}, that since $\lim_{x \to \infty} \tanh x = 1$,
a comparison for the process $R$ at large scales is a Brownian motion with a constant drift~$1/2$. 

\begin{proposition}
\label{prop:lln}
It holds that, a.s.,
\begin{equation}
    \label{eq:R-BM-approx}
\sup_{t \in\RP} \Bigl| R_t   - W_t^R - \frac{t}{2} \Bigr| < \infty, \quad \text{ and } \quad 
\lim_{t\to\infty} \frac{R_t}{t} = \lim_{t \to \infty} \frac{\EE R_t}{t} = \frac{1}{2}. \end{equation}
  Moreover, there exists a random variable
  $\theta_\infty$, with $\theta_\infty \sim \unif [0,2\pi)$, 
  and, for every $s >0$, a random variable $\Theta_\infty^{(s)} \in \RR$, such that
    \begin{equation}
    \label{eq:angular-convergence}
     \lim_{t \to \infty} \theta_t = \theta_\infty, \as, \quad \text{ and, for every $s >0$, }\quad 
     \lim_{t \to \infty} \Theta^{(s)}_t = \Theta^{(s)}_\infty, \as
\end{equation}
Moreover, for every $s>0$, 
    \begin{equation}
    \label{eq:angular-rate}
      \limsup_{t \to \infty} t^{-1} \log \bigl| \Theta^{(s)}_\infty - \Theta^{(s)}_t \bigr|  = 
      \limsup_{t \to \infty} t^{-1} \log \bigl| \theta_\infty - \theta_t \bigr|
=  -\frac{1}{2}, \as 
\end{equation}
\end{proposition}
\begin{remark}
    \label{rem:lln}
The \emph{strong law} ($\lim_{t\to\infty} t^{-1} R_t = 1/2$), 
and the \emph{limiting direction} 
($\lim_{t \to \infty} \theta_t$ exists and is uniform) results
are both well known, and some or all elements
can be found in each of e.g.~\cite[\S V]{pinsky}, \cite[Theorem 2.3]{gruet1997}, \cite[\S V.36]{rw2}, and~\cite{shiozawa} for instance. However, we were unable to find either the 
quantification of the angular convergence in~\eqref{eq:angular-rate}, which says, roughly speaking, that
$\theta$ converges at a rate subsequentially no better than $\ee^{-t(1+o(1))/2}$, or the 
convergence of $\EE R_t$ in~\eqref{eq:R-BM-approx} in the literature. Since these parts of the result are particularly relevant for us (cf.~\eqref{eq:lower-bound} below), and to make the current paper more self-contained, we prove Proposition~\ref{prop:lln} in~\S\ref{sec:lln}. 
\end{remark}

\subsection{The convex hull and its perimeter length}
\label{sec:convex-hull}

The trajectory $B[0,t] := \{ B_s : s \in [0,t] \}$
of hyperbolic Brownian motion up to time $t \in \RP$
is a random compact subset of $\HH^2$ containing the origin~$o$; the subject of our work is
the \emph{closed convex hull}
$\Hh_t := \conv B[0,t]$,  for $t \in \RP$, 
the smallest closed convex subset of $\HH^2$ containing $B[0,t]$.
Recall that a set $A \subseteq \HH^2$ is convex if and only if for every pair of distinct points $x,y \in A$, the (unique) hyperbolic geodesic between $x$ and $y$ is contained in~$A$. Then $\conv S$, for $S \subseteq \HH^2$, is the intersection
of all closed convex sets $A \subseteq \HH^2$ with $S \subseteq A$.
Since $\HH^2$ is locally convex and complete, and $B[0,t]$ is compact, $\Hh_t$ is also compact. 
See Figure~\ref{fig:hyperbolic-simulation} for a simulation.

A compact, convex set $K \subset \HH^2$ can be arbitrarily well-approximated by convex (hyperbolic) polygons (i.e., convex hulls of finitely-many points), and 
its boundary $\partial K$ has a well-defined 
\emph{perimeter length} $\perim K$.
Consider the  perimeter length of the hyperbolic Brownian convex hull:
$L_t := \perim \Hh_t$,  for $t \in \RP$;
note $L_0 = 0$. 
In this paper, we are interested in $\EE L_t$, and, particularly, in the asymptotics of $\EE L_t$ as $t \to \infty$. 
For sets $A_1 \subseteq A_2$, the monotonicity property
$\perim \conv A_1 \leq 
\perim \conv A_2$ holds
(this can be checked using, e.g., the Cauchy formula of~\cite{abf} given at~\eqref{eq:PeriminD} in \S\ref{sec:cauchy} below).
The convex hull $\Hh_t$  certainly contains the geodesic line $\conv \{ o, B_t \}$ between the origin and the location of the process at time~$t$, and hence
its perimeter is at least twice the hyperbolic distance from the origin
(this also follows from the Cauchy formula, as we explain in Example~\ref{ex:line-segment}). In other words, it holds that
$L_t \geq 2 R_t$ for all $t \in \RP$. An immediate consequence of Proposition \ref{prop:lln} is thus
that the perimeter length satisfies the ``line segment'' lower bounds
    \begin{equation}
\label{eq:lower-bound}
     \liminf_{t \to \infty} \frac{L_t}{t} \geq 1 , \as, \quad \text{ and } \quad\liminf_{t \to \infty} \frac{\EE L_t}{t} \geq 1. \end{equation}

Our 
main result gives an expression for $\EE L_t$ in terms of
an \emph{exponential functional} of Brownian motion on the line, as well as large-time and small-time asymptotics for  $\EE L_t$. Define
\begin{equation}
\label{eq:xi-clean}
\Ee_t := \int_0^t \exp ( 2 W_s - s) \ud s, \text{ for } t\in \RP, \end{equation}
where $W = (W_t)_{t \in \RP}$ is a standard Brownian motion on~$\RR$. 

\begin{theorem}
    \label{thm:expectation}
    It holds that
    $\EE L_t = \sqrt{8\pi} \, \EE  \sqrt{ \Ee_t} $ for every $t \in \RP$. Moreover,
        \begin{equation}
            \label{eq:expectation-asymptotics} 
       \lim_{t \to 0} \frac{\EE L_t}{\sqrt{ 8 \pi t}} = 1, \quad \text{ and }\quad
        \lim_{t \to \infty} \frac{\EE L_t}{2t} = 1 .\end{equation}
\end{theorem}

Theorem~\ref{thm:expectation} is a direct consequence of Theorem~\ref{thm:expectation-x-representation} combined with Proposition~\ref{prop:exp-max-X} below.

\begin{remark}
    \label{rem:expectation}
The large-time asymptotics in~\eqref{eq:expectation-asymptotics} of 
Theorem~\ref{thm:expectation} show that the ``line segment'' lower bound from~\eqref{eq:lower-bound} is not sharp, by a factor of~$2$, at least in the expectation sense. 
On the other hand, the line segment bound is of the correct order, 
which indicates the delicate balance between the exponential convergence of $\theta_t$
given in~\eqref{eq:angular-rate} and the exponential growth of hyperbolic arc-length with increasing radius. 
The small-time asymptotics
in~\eqref{eq:expectation-asymptotics} coincide with those in the Euclidean setting~\cite{letac-takacs} in accord with the intuition that Brownian motion experiences hyperbolic space as locally flat: see~\eqref{eq:euclidean-formula} and the subsequent comments below.
\end{remark}

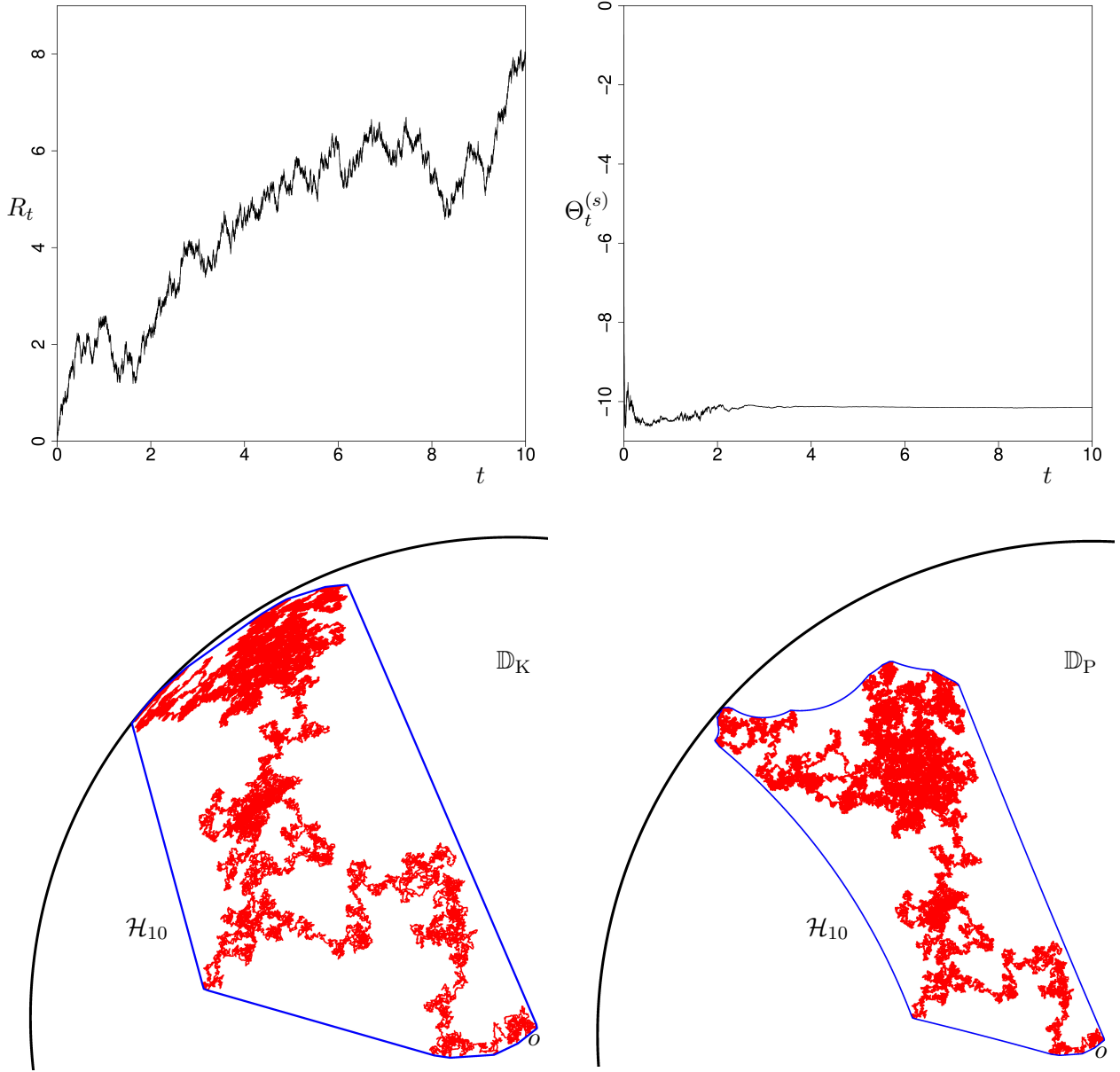
\begin{figure}
    \centering
    \begin{tikzpicture} 
    \node at (0,8.5) {\includegraphics[width=0.47\textwidth]{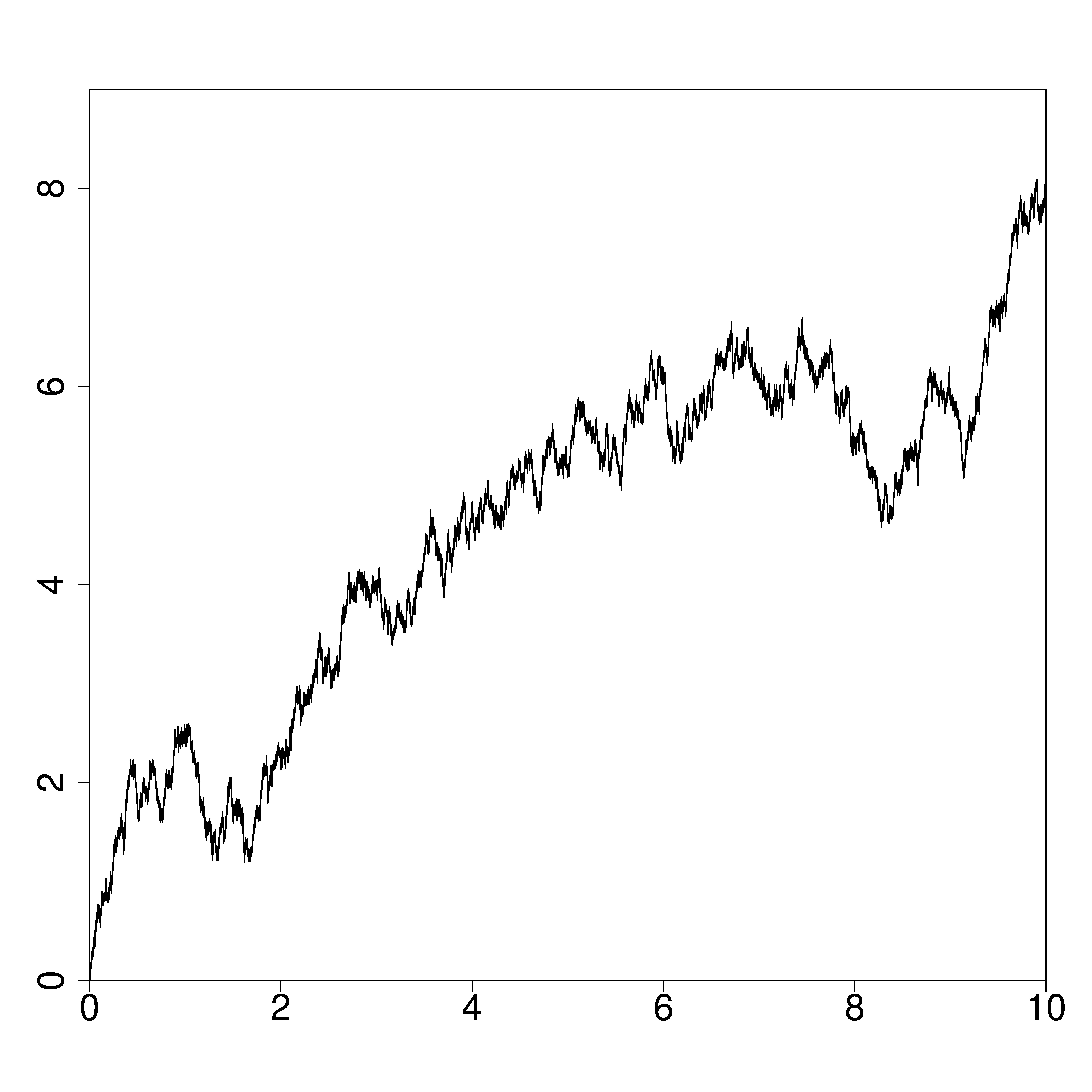}}; 
    \node at (8.5,8.5)  {\includegraphics[width=0.47\textwidth]{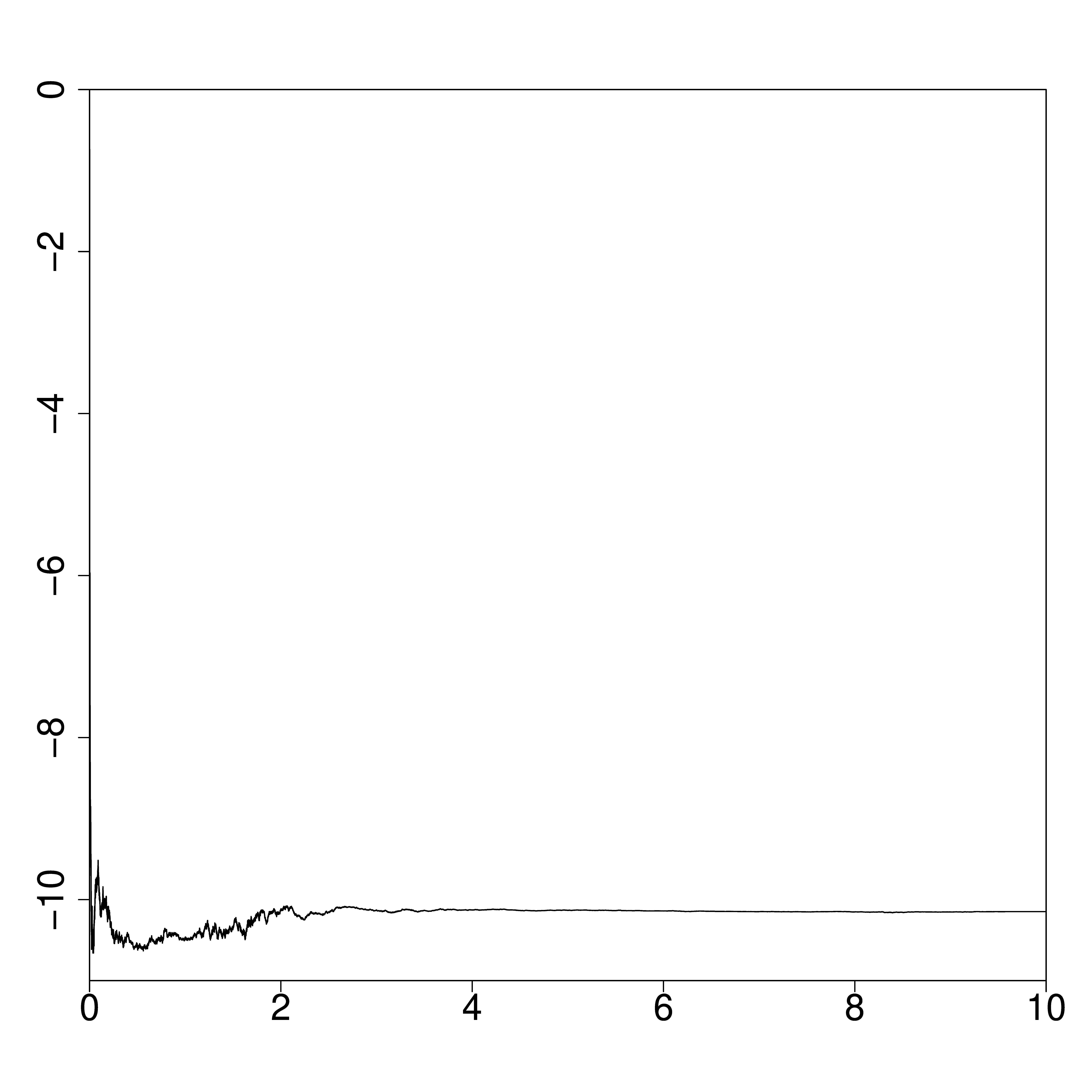}}; 
        \node at (0,0) {\includegraphics[width=0.47\textwidth]{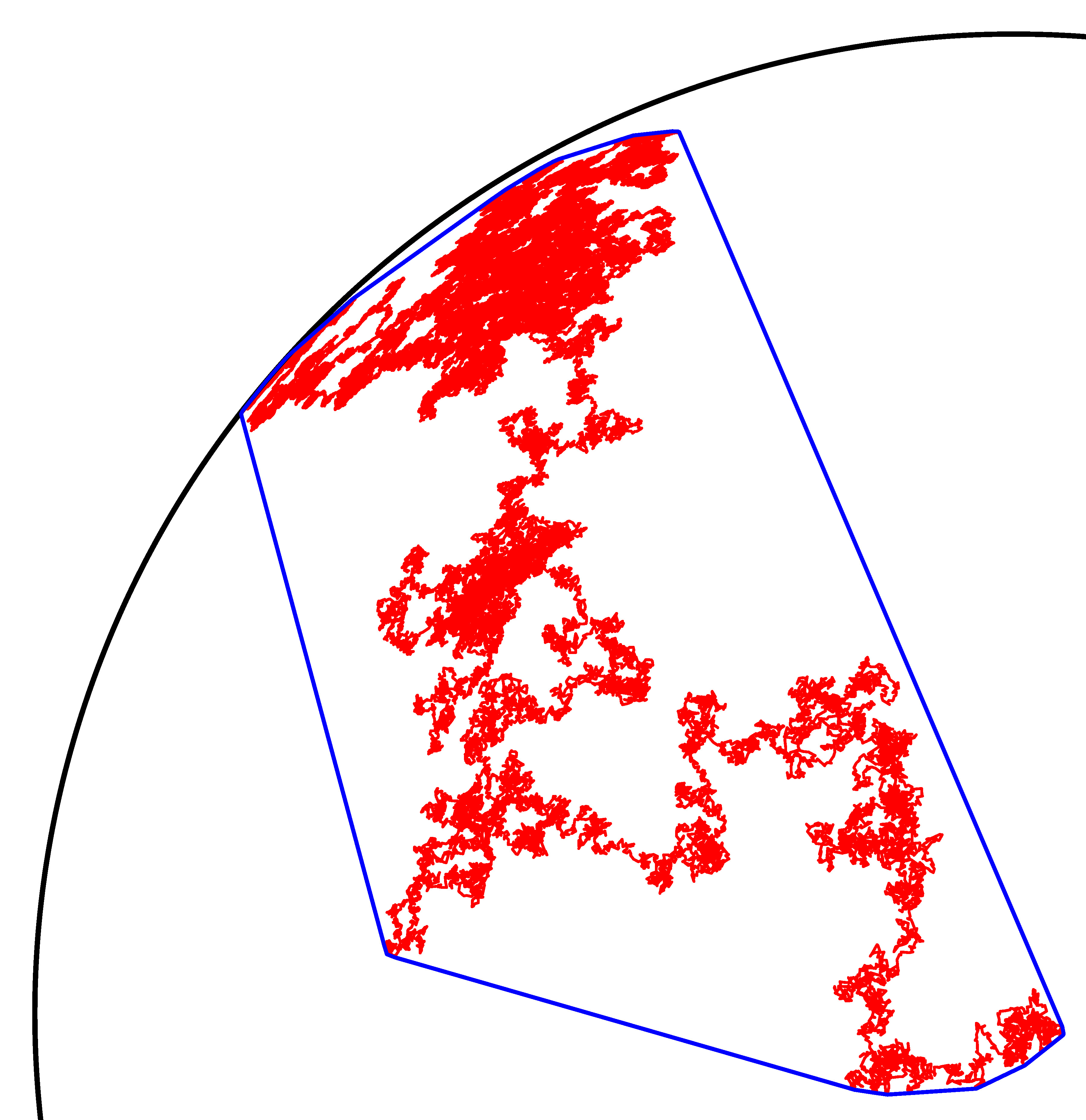}}; 
    \node at (8.5,0)  {\includegraphics[width=0.47\textwidth]{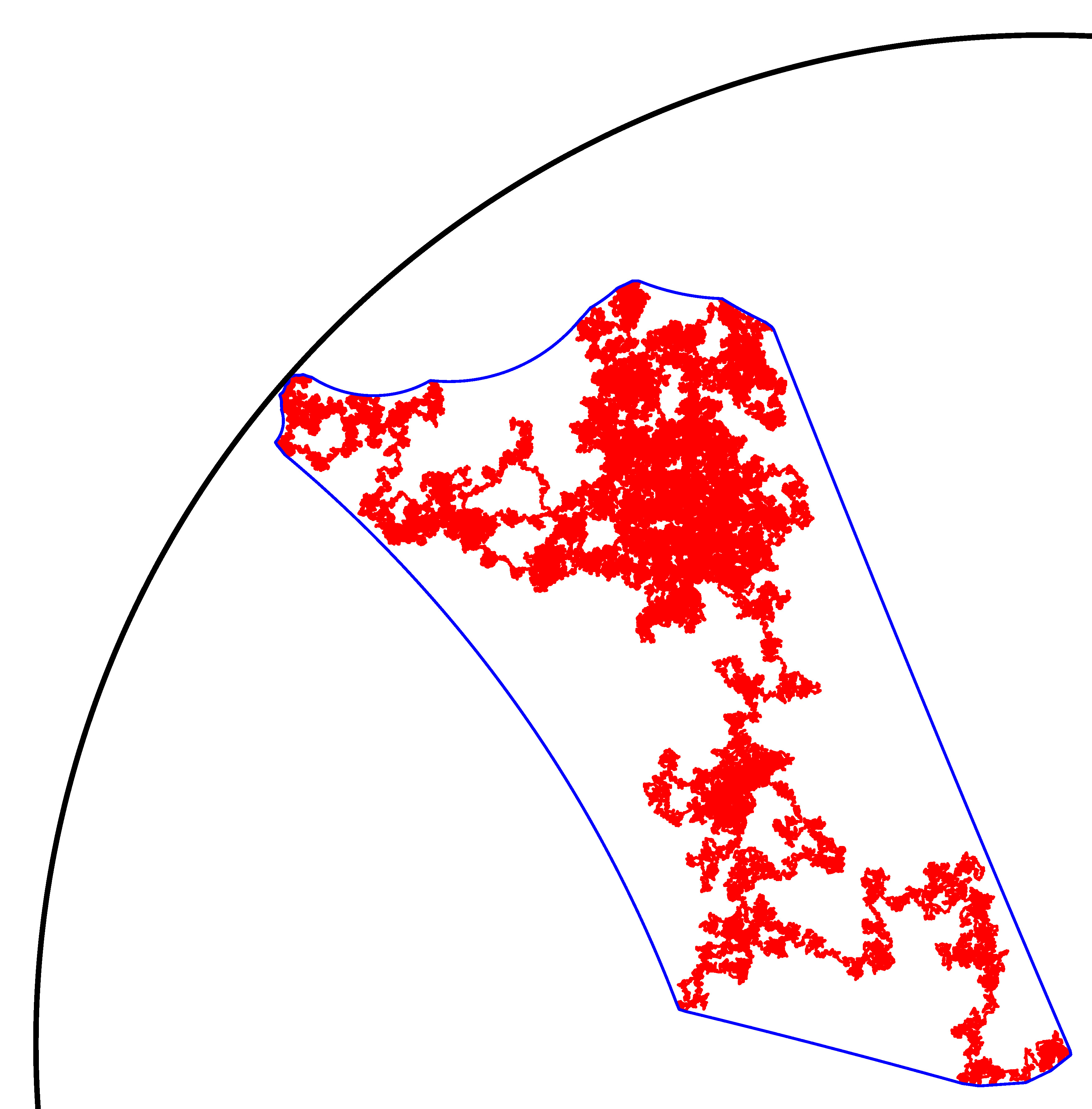}}; 
    \node at (3.5,2) {$\DDK$};
     \node at (-2,-2) {$\Hh_{10}$};
      \node at (12,2) {$\DDP$};
            \node at (3.8,-3.67) {$o$};
                  \node at (12.3,-3.85) {$o$};
     \node at (8.2,-2) {$\Hh_{10}$};
      \node at (3,4.8) {$t$};
      \node at (11.5,4.8) {$t$};
       \node at (-3.9,8.8) {$R_t$};
       \node at (4.6,8.8) {$\Theta^{(s)}_t$};
        \end{tikzpicture} 
    \caption{Simulated hyperbolic Brownian motion trajectory and its convex hull. The SDEs~\eqref{eq:r-sde} and~\eqref{eq:theta-sde} were each approximately solved, in turn, over time interval $[0,10]$ with an Euler scheme with $10^6$ steps: \emph{top left pane} shows the $R$ process  and \emph{top right pane} shows the $\Theta^{(s)}$ process over time $[s,10]$ for $s = 10^{-3}$. The discrete approximation moving swiftly to a large negative value reflects the rapid spinning out from the origin. The \emph{bottom left pane} shows the resulting Brownian trajectory $B_t$ (shown in red) over time $t \in [0,10]$ in a section of the Beltrami--Klein disk~$\DDK$;
    in this model geodesics are straight lines, so the convex hull $\Hh_{10}$ (boundary in blue) can be computed more easily. The \emph{bottom right pane} shows that same trajectory (red) and convex hull (blue) represented in the Poincar\'e disk~$\DDP$. See~\S\ref{sec:models} for a summary of the different models of $\HH^2$ and how one maps between them. Note that already by time $10$ the process appears very close to the boundary, so running longer simulations would yield little more visual information.}
    \label{fig:hyperbolic-simulation}
\end{figure}

There has been a lot of work in the last couple of decades on 
 exponential functionals such as $\Ee_t$, motivated in part
by actuarial or financial applications~\cite{d90,dsg}: 
see~\cite{my05} for a survey, 
see e.g.~\cite{yor92,ag} for their links to hyperbolic Brownian motion, and
e.g.~\cite{yor92,yor01,hy2004} for further properties. The 
asymptotics available for
moments of $\Ee_t$~\cite{hy2004} are the source of~\eqref{eq:expectation-asymptotics}.
On the other hand, the expectation identity we establish in Theorem~\ref{thm:expectation} combined with results
of~\cite{yor92} yields 
for $\EE L_t$ at each fixed $t \in \RP$
an exact, but complicated, multiple-integral expression~\eqref{eq:exact-formula-yor}, and it is not straightforward to even obtain a numerical estimate for~$\EE L_1$ (see~Appendix~\ref{sec:exact} below).
However, a remarkable identity due to Yor~\cite{yor01} leads to an exact expression at a random \emph{exponential} time;
let $\expo{\lambda}$ denote the exponential distribution with parameter $\lambda >0$ (hence mean $1/\lambda$).

\begin{corollary}
\label{cor:yor}
For $\lambda >0$, let $T_\lambda \sim \expo{\lambda}$ be  independent of $B$. Then
    \begin{equation}
    \label{eq:cor-yor}
    \EE L_{T_\lambda} = G \bigl( \sqrt{8 \lambda +1} \bigr), 
    \text{ where }
  G (x ) :=  \pi \left( \frac{x-1}{x+1} \right) \left( \frac{\Gamma \bigl( \tfrac{x-1}{4} \bigr)}{ \Gamma \bigl( \tfrac{x+1}{4} \bigr)} \right)^2
    , \end{equation}
    where $\Gamma$ denotes the (Euler) gamma function. 
   In particular, we have $\lim_{\lambda \to \infty} \sqrt{\lambda} \EE L_{T_\lambda} = \pi \sqrt{2}$, 
   $\lim_{\lambda \to 0} \lambda \EE L_{T_\lambda} = 2$, 
    and, 
   when $\lambda =1$, $\EE L_{T_1} = G(3) = \pi^2/2    \approx 4.93\ldots$.
\end{corollary}

In the case of Euclidean Brownian motion on~$\RR^2$, its perimeter
length $L^{\mathrm{E}}_t$ enjoys the exact formula, due to Letac
and Tak\'acs~\cite{letac-takacs},
\begin{equation}
    \label{eq:euclidean-formula}
    \EE L^{\mathrm{E}}_t = \sqrt{ 8 \pi t} , \text{ for all } t \in \RP,
\end{equation}
from which it readily follows that if $T_\lambda \sim \expo{\lambda}$ is independent of the Brownian motion, then 
\[ \EE L^{\mathrm{E}}_{T_\lambda} = \sqrt{8\pi} \, \EE \sqrt{T_\lambda} = \pi \sqrt{ 2 / \lambda}, \text{ for } \lambda > 0;\]
in particular, when $\lambda =1$, $\EE  L^{\mathrm{E}}_{T_1} = 
\pi \sqrt{2} \approx 4.44\ldots$.
The appearance of~$\sqrt{8\pi}$ in both Theorem~\ref{thm:expectation} and in~\eqref{eq:euclidean-formula}
is not merely a coincidence, but a reflection of the local-Euclidean heuristic at small times. This is in line with the fact that the heat kernel $p_t(x,y)$ of $\HH^2$ has the small-time behaviour (which holds in general, see e.g. \cite[Chapter 5]{hsu}), recalling that $d_\HH$ denotes the hyperbolic distance, 
$$
\lim_{t\to0} -2t\log p_t(x,y) = d_\HH(x,y)^2,
$$
which resembles its Euclidean counterpart especially when $x$ and $y$ are close;
also taking $t \approx 0$ in the SDE description yields the same intuition, since, as $r\to 0$, $\tanh r \sim r$ and $\sinh r \sim r$, so~\eqref{eq:r-sde}--\eqref{eq:theta-sde} approximate the polar description of Brownian motion on $\RR^2$ for $t \approx 0$. 
 See~\S\ref{sec:comments} below for further comparison between the hyperbolic and Euclidean settings.

Theorem~\ref{thm:expectation} is proved in~\S\S\ref{sec:hypBM}--\ref{sec:perimeter}.
To prepare for the proofs, in \S\ref{sec:hyperbolic} we recall some appropriate hyperbolic geometry, introduce notation associated with models of $\HH^2$ that we will refer to later, and present a Cauchy formula from~\cite{abf} that is a key tool. Then \S\ref{sec:hypBM} turns to hyperbolic Brownian motion on the Poincar\'e half-plane model, including its SDE description (\S\ref{sec:SDE}), its asymptotics (\S\ref{sec:half-plane-limit}) and the relation to exponential functionals (\S\ref{sec:exponential}). As mentioned above, the proof of Proposition~\ref{prop:lln} is given in \S\ref{sec:lln}; this section is essentially self-contained, using only the SDE~\eqref{eq:r-sde}, and can be read independently of~\S\S\ref{sec:hyperbolic}--\ref{sec:perimeter}. Finally, Appendix~\ref{sec:exact} provides an exact formula for $\EE L_t$ as mentioned below Theorem~\ref{thm:expectation},
and gives the proof of Corollary~\ref{cor:yor}.

To conclude this introduction, in the next subsection (\S\ref{sec:comments}) we discuss some consequences of our main results and some open questions that they raise, and draw some contrasts between the behaviour of the hyperbolic Brownian convex hull and the case of Brownian motion in~$\RR^d$, $d \in \NN$ (hereafter referred to as the ``Euclidean case''). 
Convex hulls of Euclidean random walks, Brownian motions, and L\'evy processes (etc.) have been an active topic
of study going back to L\'evy~\cite{levy}; see~\cite{majumdar} for a survey, and also~\cite{ss,sw,chp,wx1,wx2,mw,eldan,ill,av2021,vz2018,bang} for a selection of classical and more recent work.
It is natural to investigate related phenomena for hyperbolic Brownian motion 
in view of recent interest in ``hyperbolic stochastic geometry'', e.g.~\cite{bfm,gkt,bht,hht,ot,oy}.
Other aspects of the large-scale geometry of hyperbolic Brownian motion have been studied in~\cite{gruet2001,gruet1997}.

\subsection{Discussion and open questions}
\label{sec:comments}

One important contrast between the behaviour of hyperbolic and Euclidean Brownian motions is the limiting-direction result~\eqref{eq:angular-convergence}. Clearly, $\Hh_s \subseteq \Hh_t$ for $0 \leq s \leq t$; one consequence of~\eqref{eq:angular-convergence} is that  the monotone limit set $\Hh_\infty := \cup_{t \in \RP} \Hh_t$ does \emph{not} occupy the whole space, i.e.~$\PP ( \Hh_\infty = \HH^2 ) =0$. This is in contrast to the (standard, driftless) Euclidean case, where the corresponding limit set is the whole of~$\RR^d$; note that this is a reflection of how the process explores the sphere of directions, rather than of compact-set recurrence/transience.

Euclidean Brownian motion with \emph{non-zero drift} also has the limiting direction property, and for the same reason its limit convex hull is not the full space. Indeed,  the Euclidean analogue of the strong-law behaviour in~\eqref{eq:R-BM-approx} means that the convex hull of Euclidean Brownian motion with drift resembles a line segment for large~$t$, in a sense that is strong enough 
to deduce, for example, strong laws for the perimeter length; see~\cite{ss,MW18,lmw} for laws of large numbers for the perimeter of the convex hull of Euclidean random walks with drift, and~\cite{wx1,MW18} for fluctuations. 

Despite the fact that (by Proposition~\ref{prop:lln}) hyperbolic Brownian motion also satisfies a strong law and has a limiting direction, one contribution of our main result is to show that this line-segment comparison breaks down in the hyperbolic case. 
Our  Theorem~\ref{thm:expectation} shows that for $t$ large we have $\EE L_t \approx 4\EE R_t$. This is twice as much as what one would obtain if the convex hull of hyperbolic Brownian motion was indeed getting close to a hyperbolic line segment,
corresponding to the fact that the constant in Theorem~\ref{thm:expectation}
is twice the lower bound provided by~\eqref{eq:lower-bound}.
Natural questions remain; the next would settle whether or not a law of large numbers holds for $L_t$.

\begin{problem}
Is there a distributional limit for $L_t/t$? If so, is it constant?
\end{problem}

Partial progress on the above might be to first consider:

\begin{problem}
Obtain asymptotics, or good bounds in either direction, for $\Var L_t$.
\end{problem}

To address these questions, finer study of the fluctuations of the angle process~$\theta_t$,
as partially quantified by~\eqref{eq:angular-rate} above, seems necessary (cf.~Remark~\ref{rem:expectation}).  

To give some further insight in the comparison with the Euclidean setting, we consider the scaled trajectory $u \mapsto \left(t^{-1}R_{ut},\theta_{ut}\right)$ for $u \in [0,1]$. By the triangle inequality we have
$$
d_\HH\left((t^{-1}R_{ut},\theta_{ut}),(u,\theta_\infty)\right) \leq |t^{-1}R_{ut} - u| + \sinh(u)|\theta_{ut} - \theta_\infty|.
$$
From Proposition~\ref{prop:lln} it follows that $\lim_{t \to \infty} \sup_{0\leq u\leq 1} | t^{-1}R_{ut} - u | = 0$, a.s.  Furthermore,  for $t >0$,
\begin{align*}
\sup_{0\leq u\leq 1} \sinh(u)|\theta_{ut} - \theta_\infty| 
&=
\max\left\{\sup_{0\leq u\leq t^{-1/2} } \sinh(u)|\theta_{ut} - \theta_\infty|,\sup_{t^{-1/2}\leq u\leq 1} \sinh(u)|\theta_{ut} - \theta_\infty|\right\} 
\\
&\leq
\max\left\{2\pi\sinh (t^{-1/2} ),\sinh(1)\sup_{t^{-1/2} \leq u\leq 1} |\theta_{ut} - \theta_\infty| \right\}.
\end{align*}
 Proposition~\ref{prop:lln} shows that $\theta_t \to \theta_\infty$, a.s., implying that the second term inside the maximum goes to 0. Since   the first term also goes to 0 as $t \to \infty$, we can collect everything to find that
$$
\lim_{t\to\infty} \sup_{0 \leq u \leq 1} d_\HH\left((t^{-1}R_{ut},\theta_{ut}),(u,\theta_\infty)\right) = 0.
$$
This implies that as sets,
$$
\{(t^{-1}R_{ut},\theta_{ut}) \colon 0 \leq u \leq 1\} \to \{(u,\theta_\infty) \colon 0 \leq u \leq 1\}
$$
in the Hausdorff distance. Following reasoning based on continuous mapping (cf.~\cite{lmw}), this implies 
$$
\lim_{t \to \infty} \perim\conv\{(t^{-1}R_{ut},\theta_{ut}) \colon 0 \leq u \leq 1\} = 2 \as
$$
In the Euclidean setting, from here one can derive the behaviour of the perimeter of the convex hull of $\{(R_{ut},\theta_{ut}) \colon 0 \leq u \leq 1\}$, since it is simply $t$ times larger. However, since distance in hyperbolic space grows differently, in the hyperbolic setting no such implication is valid. 

Finally, we raise:

\begin{problem}
Consider Brownian motion in higher-dimensional hyperbolic space $\HH^d$, $d  \geq 3$. Exact formulae for expected volumes and surfaces areas, in the vein of the Euclidean case~\cite{eldan}, seem out of reach, but asymptotics in the vein of our Theorem~\ref{thm:expectation} would be of interest.
For $d >2$, an appropriate analogue of Proposition~\ref{prop:lln} holds (the limiting speed is, in general, $\frac{d-1}{2}$ instead of $1/2$~\cite[p.~41]{gruet1997}) but the planar Cauchy formula from \S\ref{sec:cauchy} may be replaced
by a more involved Crofton-type formula (see e.g.~\cite[p.~907]{hht}). 
\end{problem}

\section{Background on hyperbolic geometry}
\label{sec:hyperbolic}

\subsection{Models of the hyperbolic plane}
\label{sec:models}

There are several commonly used models for the hyperbolic plane $\HH^2$,
which represent the metric space $\HH^2$ via some (subset of) Euclidean space endowed with a specific metric. Such a model is needed for performing actual computations, and the different models have different features that make certain computations more convenient in one model or the other, so it is sometimes helpful to switch among several models. The following paragraphs give a compact overview 
of what we will use in this paper; we refer to~\cite{CFKP97}  for a more detailed
account of the geometry, and acknowledge~\cite[\S V.36]{rw2} and~\cite{lo}
as containing lucid presentations with similar intent as the present section.

\paragraph{The Poincar\'e half plane model.}
The  Poincar\'e half-plane $\HHP$ takes as its base space the Euclidean upper half-plane $\RR \times (0,\infty)$. We write $(x,y) \in \HHP$
for the (Cartesian) \emph{half-plane coordinates}. Here $(0,1)$ represents the origin
of $\HH^2$.
The geodesic polar coordinates $(R,\theta) \in \RP \times [0,2\pi)$ for a point in $\HH^2$ are then related to the
Cartesian coordinates $(x,y) \in \HHP$ via the transform
\begin{equation}
    \label{eq:geodesic-polar-xy}
 x = \frac{\sinh R \cos \theta}{\cosh R - \sinh R \sin \theta} ,~~~
y = \frac{1}{\cosh R - \sinh R \sin \theta} .
\end{equation}
In particular, the hyperbolic radius $R$ of $(x,y) \in \HHP$ satisfies
\[ \cosh R = \frac{x^2 + y^2 + 1}{2y} = 1 + \frac{x^2 + (y-1)^2}{2y}. \]

\paragraph{The Poincar\'e disk model.}
The Poincar\'e disk $\DDP$ takes as its base space the open unit disk $\DD := \{ (u,v) \in \RR^2 : u^2 + v^2 < 1\}$, with a certain metric.
A point $(x,y) \in \HHP$ corresponds to the point $(u,v) \in \DDP$ via the canonical transform (mapping $(0,1)$ to $(0,0)$ and $(0,0)$ to $(0,-1)$)
\[ u = \frac{2x}{x^2 + (y+1)^2}, ~~~ v = \frac{x^2 + y^2 -1}{x^2 + (y+1)^2} .\]
Equivalently, a point $(u,v) \in \DDP$ corresponds to $(x,y) \in \HHP$ via 
\begin{equation}
    \label{eq:xy-uv} x = \frac{2u}{u^2 + (1-v)^2}, ~~~ y = \frac{1-(u^2 + v^2)}{u^2 + (1-v)^2} = \frac{2 (1-v)}{u^2 + (1-v)^2} - 1 .\end{equation}
It is also  useful to represent $(u,v) \in \DDP$ via \emph{polar disk coordinates} $(r,\theta) \in [0,1) \times [0,2\pi)$, given by
$r^2 = u^2 + v^2$, $u = r \cos \theta$, and~$v = r \sin \theta$.
Comparison of~\eqref{eq:geodesic-polar-xy} and~\eqref{eq:xy-uv} shows that  
\[ \frac{x}{y} = \sinh R \cos \theta = \frac{2 r \cos \theta}{1-r^2} , \text{ for } y >0,\]
noting that the disk polar angle coincides with the geodesic polar angle,
and hence the geodesic radius and the disk radius are related by
\[ \sinh R = \frac{2r}{1-r^2}, ~~~ \cosh R = \frac{1+r^2}{1-r^2} , \]
or, equivalently,
\begin{equation}
\label{eq:R-r-log}
R = \log \frac{1+r}{1-r} = 2 \arctanh r, ~~~ r = \frac{\ee^{R} -1}{\ee^{R} +1} = \tanh (R/2) . \end{equation}

\paragraph{The Beltrami--Klein disk.}  Like the 
Poincar\'e disk model, $\DDP$, the
Beltrami--Klein disk $\DDK$ model takes the  open unit disk $\DD$ as its base space, but uses a different Riemannian metric. Both $\DDP$ and $\DDK$ are obtained by pushing forward the metric on the hemisphere model for the hyperbolic plane through two different projections: orthogonal projection gives the Beltrami--Klein model, while stereographic projection gives the Poincar\'e disk model. The benefit of the Beltrami--Klein model is that geodesics are straight lines, making them easy to compute. However, the Riemannian metric (and hence the Laplacian) have much simpler formulas in the Poincar\'e disk model.

 To go between the two models, one simply has to rescale the points in an appropriate way. In particular, the functions $f,g:\DD \to \DD$ given by
\begin{equation}\label{e:pdtobk}
    f(z) = \frac{2z}{1 + |z|^2}, \quad\text{ and }\quad
g(z) = \frac{z}{1 + \sqrt{1 - |z|^2}}
\end{equation}
are mutually inverse, and, $f: \DDP \to \DDK$ (and equivalently $g: \DDK \to \DDP$) is an isometry (see \cite{CFKP97}, where one can obtain these function as composition of isometries via the hemisphere model). In particular,   if $B = (B_t)_{t \in \RP}$ is a hyperbolic Brownian motion represented in the Poincar\'e disk model, then  $\widetilde{ B } = ( \widetilde{ B }_t )_{t \in \RP}$ given by
$$
\widetilde{ B }_t = f(B_t) = \frac{2B_t}{1 + |B_t|^2}
$$
is a Brownian motion in the Beltrami--Klein model.

\subsection{Hyperbolic Cauchy formula}
\label{sec:cauchy}

In Euclidean integral geometry, the \emph{Cauchy formula} expresses
the perimeter of a convex body via an integral over widths of all of its one-dimensional projections. 
In the hyperbolic space $\HH^2$,  an analogous Cauchy formula for the  Beltrami--Klein disk model $\DDK$ was given by~\cite{abf}. Again, the utility of the formula is that it reduces a two-dimensional problem to a family of one-dimensional problems.

To describe the formula, consider a hyperbolic convex body $K \subset \DDK$.
Fix an angle $\varphi \in [0,2\pi)$ at the origin $(0,0)$ (with respect to the horizontal axis) in $\DDK$. 
Take point 
$R(\varphi) := (\cos \varphi, \sin \varphi)$ on the boundary of $\DDK$, 
and its anticlockwise orthogonal companion $R^\perp (\varphi) := (- \sin\varphi, \cos \varphi)$. Let $\ell (x, \varphi)$ denote the line through $R(\varphi)$ and $x \in K$ and let $\ell^\perp (\varphi)$
denote the line through $(0,0)$ and $R^\perp (\varphi)$,
parametrized by signed Euclidean arc length as $\ell^\perp (\varphi) := \{ \ell^\perp (\varphi, \lambda) : \lambda \in \RR \}$. For given $x \in K$ and $\varphi \in [0,2\pi]$,
let $\lambda (\varphi,x)$ denote the value of
$\lambda \in \RR$
such that
$\ell^\perp (\varphi, \lambda) \in \ell (x, \varphi)$,
i.e., the parameter corresponding to the intersection point of lines
 $\ell (x, \varphi)$ and $\ell^\perp (\varphi)$.
 Then the Cauchy formula (see~equation~(1.6) of \cite{abf}) says that
  the hyperbolic length of the perimeter of 
the hyperbolic convex body $K \subset \DDK$ is given by
 \begin{equation}
\label{eq:cauchy}
\perim K = \int_0^{2\pi} \sup_{x \in K} \lambda (\varphi,x) \ud \varphi.\end{equation}
See also Figure~2 in~\cite[p.~1829]{abf} for an illustration. 
A calculation shows that
\begin{equation} \label{eq:lambda}
\lambda (\varphi,x) = \frac{x_2 \cos \varphi - x_1 \sin \varphi}{1-x_1 \cos \varphi - x_2 \sin \varphi} ,
\end{equation}
where $x=(x_1,x_2) \in \DD$ are the Cartesian coordinates of the point $x \in \DDK$. 

Note that if $K = \conv(\tilde K)$ for some $\tilde K \subset \DDK$, it suffices to take the supremum in \eqref{eq:cauchy} over $\tilde K$ instead. When $K$ is a straight line, this is immediate from the construction. The general case then follows, because any two points in $\conv(\tilde K)$ lie on a straight line connecting points in $\tilde K$. 

We will sometimes wish to apply the Cauchy formula in $\DDP$ rather than $\DDK$.
For a convex body $K'\subset \DDP$, by transforming
via \eqref{e:pdtobk}, one obtains a convex body $f(K') \subset \DDK$ since isometries map geodesics to geodesics. Consequently, we can obtain from \eqref{eq:cauchy} that
\begin{equation}
\label{eq:PeriminD}
\perim K' = \int_0^{2\pi} \sup_{x \in K'} \lambda \left(\varphi,\frac{2x}{1+|x|^2}\right) \ud \varphi.
\end{equation}

We will use a modified version of \eqref{eq:PeriminD}, where we take the supremum over a set $\hat K$ such that $K' = \conv(\hat K)$. This follows from the same claim above for \eqref{eq:cauchy}, because $f(K') = \conv(f(\hat K))$, which holds because $f$ maps geodesics to geodesics.  

The following example demonstrates the application of the Cauchy formula~\eqref{eq:cauchy} for a hyperbolic line segment; recall from~\eqref{eq:lower-bound} that twice the length of the geodesic line segment between $B_0 = o$ and $B_t$ provides a lower bound for the perimeter length $L_t$, and so this example also serves as a comparison for what comes later. 

\begin{example}[Line segment]
\label{ex:line-segment}
   Consider the line segment between the origin
   and the point $x_{h,\theta} \in \HH^2 \setminus \{o\}$ represented in $\DDK$ by the Euclidean  polar angle $\theta$ and radius $h \in (0,1)$, i.e.,
   $K_{h,\theta} := \conv \{ o , x_{h,\theta} \}$.
The hyperbolic distance $d_\HH (o, x_{h,\theta})$ 
corresponds in   $\DDK$ to the 
distance from the origin to a point at radius $h$ via the formula
\[ d_\HH (o, x_{h, \theta}) = \frac{1}{2} \log \frac{1+h}{1-h} = \arctanh h ; \]
notice the factor of $1/2$ in the last display compared to the corresponding formula~\eqref{eq:R-r-log}
in~$\DDP$. 
We apply the Cauchy formula~\eqref{eq:cauchy} to compute $\perim ( K_{h,\theta})$. Note that
\begin{align*}
\sup_{x \in K_{h,\theta}} \lambda (\varphi,x)
& = \sup_{u \in [0,h]} \frac{u \sin \theta  \cos \varphi - u \cos \theta  \sin \varphi}{1- u \cos \theta  \cos \varphi - u \sin \theta  \sin \varphi} \\
& =  \sup_{u \in [0,h]}\frac{u \sin ( \theta - \varphi)}{1 - u \cos ( \theta - \varphi )} .
\end{align*}
Some calculus shows that
\[ \sup_{x \in K_{h,\theta}} \lambda (\varphi,x) = \frac{h \sin ( \theta - \varphi)}{1 - h \cos ( \theta - \varphi )}  \1 { \sin ( \theta - \varphi ) >0 }.
\]
For simplicity, we take $\theta=\pi$, so that $\sin (\theta-\varphi) >0$ over $\varphi \in (0,\pi)$. 
Then from~\eqref{eq:cauchy} we get
\begin{align*}
\perim (K_{h,\pi}) & = 
\int_0^{\pi} \frac{h \sin ( \varphi)}{1 - h \cos (  \varphi )} \ud \varphi  
= \log \frac{1+h}{1-h} = 2 d_\HH ( o, x_{h,\pi} ). \end{align*}
That is, the perimeter length of the line segment is twice the distance between its endpoints.
\end{example}

\begin{remark}
    \label{rem:radius-line-segment}
   It follows from Example~\ref{ex:line-segment} that $\perim \conv \{ o, B_t \} = 2R_t$. We will use this fact later (Remark~\ref{remark:perim_BM_linesegment} below) to show how we can use the computations of \S\S\ref{sec:hypBM}--\ref{sec:perimeter} to give a geometric approach to the asymptotics of $\EE R_t$ from Proposition~\ref{prop:lln} derived by analytic means in~\S\ref{sec:lln}.
\end{remark}

\section{Hyperbolic Brownian motion on the Poincar\'{e} half-plane}\label{sec:hypBM}

\subsection{Overview}
\label{sec:SDE}

The hyperbolic Brownian motion $B$ is given in $\HHP$ in coordinates as $(X,Y)$
for horizontal process $X = (X_t)_{t \in \RP}$ and vertical process $Y = (Y_t)_{t \in \RP}$,
and $B_0 = o$ corresponds to starting from $(X_0, Y_0 ) = (0,1)$.
For our purposes, the most important consequence of the Cauchy formula of~\S\ref{sec:cauchy}
is the following:

\begin{theorem}
\label{thm:expectation-x-representation}
For every $t \in \RP$, it holds that $\EE L_t = 2\pi \EE \bigl[ X_t^\star \bigr]$,
where $X_t^\star := \sup_{0 \leq s \leq t} X_s$ and $X = (X_t)_{t \in \RP}$ is
the horizontal process of hyperbolic Brownian motion in the Poincar\'{e} half-plane~$\HHP$ specified by~\eqref{eq:BM-xy} below.
\end{theorem}

The proof of Theorem~\ref{thm:expectation-x-representation} is the subject of~\S\ref{sec:perimeter}.
In the present section we describe $X_t^\star$, the main object in the theorem,
and establish, in Proposition~\ref{prop:exp-max-X} below, results on $\EE  [ X_t^\star  ] $
that when combined with  Theorem~\ref{thm:expectation-x-representation}, yields both the expectation identity and asymptotics in Theorem~\ref{thm:expectation}. We start with the description of   $X_t^\star$.

A hyperbolic Brownian motion on the Poincar\'{e} half-plane $\HHP$ can be defined
by the following pairs of coupled SDEs (see~\cite[\S 7]{yor92},~\cite[eq.\ (2.3)]{IM} or~\cite[eq.~(4.3)]{lo}):
\begin{equation}
    \label{eq:BM-xy}
\ud X_t = Y_t \ud W_t^X \quad \text{and} \quad \ud Y_t=Y_t \ud  W_t^Y,
\end{equation}
where $W=(W^X,W^Y)$ is a 2-dimensional standard Brownian motion.
In this model, a manifestation of the ``limiting direction''
result from Proposition~\ref{prop:lln} is that 
$\lim_{t \to \infty} Y_t = 0$ and
$\lim_{t \to \infty} X_t = X_\infty$, a.s., for some
random $X_\infty \in \RR$ (we
explain this in~\S\ref{sec:half-plane-limit} below, see \eqref{eq:limit-direction}).
Thus the maximal random variable $X_t^\star = \sup_{0 \leq s \leq t} X_s$ defined in Theorem~\ref{thm:expectation-x-representation} satisfies $\lim_{t \to \infty} X_t^\star < \infty$, a.s.
Nevertheless, it turns out that $\EE [ X_t^\star ]$ grows linearly in $t$,
as contained in the following result.

\begin{proposition}
\label{prop:exp-max-X}
Suppose that $(X,Y)$ is the Poincar\'{e} half-plane representation of a hyperbolic Brownian motion, i.e., satisfying SDEs~\eqref{eq:BM-xy} and started from $(X_0,Y_0) = (0,1)$. Then
\begin{equation}
\label{eq:X-max}
\EE \bigl[ X_t^\star \bigr]  = \sqrt{2/\pi} \EE \bigl[  \sqrt{ \Ee_t} \bigr] , \text{ for all } t \in \RP, \end{equation}
where $\Ee_t$ is defined at~\eqref{eq:xi-clean}. 
Moreover, we have the asymptotics 
\begin{equation}
\label{eq:X-max-answer}
\lim_{t \to 0} t^{-1/2} \EE \bigl[ X_t^\star \bigr]  = 
 \sqrt{2 / \pi} , \quad\text{ and }\quad
\lim_{t \to \infty} t^{-1} \EE \bigl[ X_t^\star \bigr]  = 
 1/\pi .
\end{equation}  
\end{proposition}

Proposition~\ref{prop:exp-max-X} is proved in~\S\ref{sec:half-plane-limit} using
some facts about exponential functionals of Brownian motion that we defer to~\S\ref{sec:exponential}.

\subsection{Asymptotics for the half-plane process}
\label{sec:half-plane-limit}

The 
$Y$ SDE in~\eqref{eq:BM-xy}
is autonomous, and the integral representation of the process with initial values $X_0, Y_0$ is given by~\cite[eq.\ (2.4)]{IM} 
\begin{align*}
    X_t = X_0 +\int_0^t Y_s \ud W_s^X, \text{ and } 
    Y_t = Y_0 \exp \bigl( W_t^Y -t/2 \bigr).
\end{align*}
Consider the martingale $X := (X_t)_{t \in \RP}$. 
By It\^o's formula,
\[ \ud (X_t^2) = 2 X_t Y_t \ud W_t^X + Y_t^2 \ud t ,\]
which verifies the quadratic variation process $[X] := ([X]_t)_{t \in \RP}$
associated with $X$ as 
\begin{equation}
    \label{eq:X-qv}
 [X]_t = \int_0^t Y_s^2 \ud s = Y_0^2 \xi_t, \text{ where }
 \xi_t := \int_0^t \exp \bigl( 2W_s^Y - s \bigr) \ud s  \eqd \Ee_t, 
 \end{equation}
with $\Ee_t$ defined at~\eqref{eq:xi-clean}. 
Here $W^Y$ satisfies the law of the iterated logarithm~\cite[p.~112]{KS},
  meaning that, for example, for every $c \in (0,1)$, a.s., for all $t$ sufficiently large,
\begin{equation}
\label{eq:integrand-bound} 0 \leq \exp \bigl( 2W_t^Y - t \bigr) \leq \ee^{-ct} .\end{equation}
Hence the increasing process $[X]_t$ has limit
\[ [X]_\infty := \lim_{t \to \infty}
[X]_t = Y_0^2 \xi_\infty = Y_0^2 \int_0^\infty \exp \bigl( 2W_s^Y - s \bigr) \ud s, \]
which is a.s.-finite, by~\eqref{eq:integrand-bound}. 

The Dambis--Dubins--Schwarz theorem~\cite[pp.~174--5]{KS} says that every real-valued continuous martingale, vanishing at~$0$,
is a time-change of a Brownian motion.
Hence there exists a standard Brownian motion $\tW$  on $\RR$ (at least over time interval $[0,[X]_\infty]$) for which
$X_t = X_0 + \tW_{[X]_t}$,  for all $t \in \RP$.
In general, the $\tW$ and the $[X]$ processes are not independent,
but in our case they \emph{are}, since the $W^X$ and $W^Y$ are independent:
see~\cite[p.~530]{yor92}. To summarize:

\begin{lemma}
\label{lem:time-change}
Suppose that $X_0 =0$, $Y_0 =1$.    We have the representation
$X_t = \tW_{[X]_t}=\tW_{\xi_t}$, $t \in \RP$, where $\xi$ is the exponential integral
defined at~\eqref{eq:X-qv}, and $\tW$ is Brownian motion independent of $\xi$.
\end{lemma}

Because it holds that $\xi_\infty = [X]_\infty < \infty$ (when $Y_0=1$), a.s., we get the convergence 
\begin{equation}
    \label{eq:limit-direction}
 X_\infty := \lim_{t \to \infty} X_t =  \tW_{[X]_\infty}  , \as, \text{ and } \lim_{t \to \infty} Y_t = 0, \as, \end{equation}
 which is the ``limiting direction'' from~\eqref{eq:angular-convergence}
 expressed in this model; 
 see e.g.~\cite[\S4.4]{d90} or~\cite[p.~30]{ag}.

\begin{proof}[Proof of Proposition~\ref{prop:exp-max-X}]
Since $t \mapsto [X]_t$ is continuous and increasing,
\[ X_t^\star = \sup_{0 \leq s \leq t} X_s= \sup_{0 \leq s \leq [X]_t} \tW_s ,\]
and so using the independence in Lemma~\ref{lem:time-change} and the fact that $\EE \sup_{0 \leq s \leq t} \tW_s =\sqrt{ 2t/\pi}$ (which can be obtained by an application of the reflection principle~\cite[p.~96]{KS})
we get~\eqref{eq:X-max},
recalling that $\xi_t$ has the same law as $\Ee_t$ defined at~\eqref{eq:xi-clean}. 
Asymptotics for the expectation on the right-hand side of~\eqref{eq:X-max} are delicate, but have been obtained by~\cite{hy2004}. In particular, in Lemma~\ref{lem:hy-moments}\eqref{lem:hy-moments-ii} 
in~\S\ref{sec:exponential} below 
we quote the result that $\EE [ \sqrt{ \xi_t} ] = (2 \pi)^{-1/2} t + o(t)$ as $t \to \infty$.
    Combined with~\eqref{eq:X-max}, this then yields the claimed $t \to \infty$ asymptotics for $\EE [ X_t^\star]$ in~\eqref{eq:X-max-answer}. 
    
    The $t \to 0$ asymptotics are more elementary. Indeed, writing $\overline{W}_t := \sup_{0 \leq s \leq t} |W_s|$, for every $\varepsilon>0$ we deduce from~\eqref{eq:xi-clean} the a.s.-bounds
    \[ (1-\varepsilon) t \ee^{-2 \overline{W}_t} \leq \ee^{-2 \overline{W}_t}   \int_0^t \ee^{-s} \ud s \leq  \Ee_t \leq \ee^{2 \overline{W}_t}   \int_0^t \ee^{-s} \ud s \leq  t \ee^{2 \overline{W}_t} , 
    \text{ for all } t \in [0, t_\varepsilon],
    \]
  for some small enough $t_\varepsilon > 0$ (a deterministic constant). Using the Gaussian tail bound for $\overline{W}_t$ it is not hard to show that $\lim_{t \to 0} \EE \ee^{\overline{W}_t} = \lim_{t \to 0} \EE \ee^{-\overline{W}_t} = 1$, and so, since $\varepsilon >0$ was arbitrary, we get 
    $\lim_{t\to 0} t^{-1/2} \EE \sqrt{ \Ee_t} = 1$.  Combined with~\eqref{eq:X-max}, this gives the  $t \to 0$ asymptotics in~\eqref{eq:X-max-answer}. 
\end{proof}

In the preceding proof we appealed to expectation asymptotics from~\cite{hy2004}
for the process $\xi$ defined through~\eqref{eq:X-qv}. Although not directly relevant for our main story, in the next section we take a more detailed look at $\xi$ and some of its properties.

\subsection{Exponential functionals of Brownian motion}
\label{sec:exponential}

The process 
$\xi$ defined at~\eqref{eq:X-qv} 
has the same law as the process $\Ee$ defined at~\eqref{eq:xi-clean}
in~\S\ref{sec:convex-hull}, and in this subsection we work with the latter.
Indeed, the random variable $\Ee_t$ is the object that~\cite{yor92,hy2004} call $A_t^{(-1/2)}$ (see~Appendix~\ref{sec:exact} below for the more general notation). 
Collecting parts (v), (iv), and (iii), respectively, of Theorem~2.2 in \cite{hy2004}
(in their notation, take $a=0$, $\xi=1$, $-m=p>0$, and $\mu =-1/2$), we have the following.

\begin{lemma}[Hariya \& Yor, 2004]
\label{lem:hy-moments} 
\begin{enumerate}[(i)]
\item 
 \label{lem:hy-moments-i} 
For $p \in (0,1/2)$,
    \[ \lim_{t \to \infty} \EE \bigl[  \Ee_t ^p \bigr] = \pi^{-1/2} 2^{1-p} \Gamma \left( \tfrac{1}{2} -p \right) .\]
    \item
    \label{lem:hy-moments-ii} 
    It holds that
    $\lim_{t \to \infty} t^{-1} \EE \bigl[ \sqrt{ \Ee_t} \bigr] =({2\pi} )^{-1/2}$.
    \item
     \label{lem:hy-moments-iii} 
    For $p >1/2$,
    \[ \lim_{t \to \infty} \ee^{ - p (2 p- 1 )t } \EE \bigl[ \Ee_t^p \bigr] =  2^{-p} \frac{\Gamma\left(p-\tfrac{1}{2}\right)}{\Gamma \left(2p-\tfrac{1}{2}\right)}   .\]
\end{enumerate}
\end{lemma}

Coming back to the process $X$, Lemma~\ref{lem:hy-moments}\eqref{lem:hy-moments-i} and \eqref{lem:hy-moments-ii}  imply, by monotone convergence and the fact that $\xi_\infty = [X]_\infty$ (when $Y_0=1$),  
$\EE ( [X]^{1/2}_\infty ) = \infty$,  but $\EE ( [X]^q_\infty ) < \infty$  for every $q \in [0,1/2)$. 
In fact, from an identity in law due to
Dufresne~\cite[Proposition~4.4.4]{d90} (see also~\S II.2 of~\cite{ag}),  
\begin{equation}
    \label{eq:qv-limit}
    [ X]_\infty \eqd \frac{1}{2Z_{1/2}} \eqd S_{1/2} ,
\end{equation}
where $Z_{1/2}$ is Gamma(1/2) distributed, i.e., has density proportional to $z^{-1/2} \ee^{-z}$ on $\RP$,
and thus $S_{1/2}$ has the positive $1/2$-stable (L\'evy) distribution, having density proportional to $s^{-3/2} \ee^{-1/(2s)}$ on $\RP$.
Consequently, in the ``limiting direction'' result~\eqref{eq:limit-direction}, the variable $X_\infty = W_{[X]_\infty}$ (``randomized Gaussian'')
turns out to have the Cauchy distribution. Indeed, a calculation gives
\[ \PP (X_\infty \leq x ) = \frac{1}{\Gamma ( 1/2)} \int_0^\infty \ee^{-z} z^{-1/2} \Phi \bigl( x \sqrt{2z} \bigr) \ud z , \]
where $\Phi$ is the standard normal distribution function. It follows that $X_\infty$ has density
\begin{equation}
\label{eq:cauchy-density}
\int_0^\infty \ee^{-z(1+x^2)} \ud z = \frac{1}{\pi (1+x^2)}, \text{ for } x \in \RR. 
\end{equation}
The fact that $X_\infty$ is Cauchy was already observed in~\cite[p.~1338]{cm},
where the density~\eqref{eq:cauchy-density} is called ``Lorentzian''.
To summarize:

\begin{lemma}
\label{lem:cauchy}
The random variable $[X]_\infty$ has a positive $1/2$-stable distribution,
    while $X_\infty$ has the Cauchy distribution on~$\RR$.
    In particular, $\EE [ | X_\infty |^p ] < \infty$ if and only if $p<1$.
\end{lemma}

\section{The expected perimeter and its asymptotics} 
\label{sec:perimeter}

We are now in a position to give the proof of Theorem~\ref{thm:expectation-x-representation}, and hence our main result, Theorem~\ref{thm:expectation}. Below, we consider the hyperbolic Brownian motion $(X_t,Y_t)$ for $t \in \RP$ on the Poincar\'{e} upper half-plane~$\HHP$ from~\S\ref{sec:hypBM} started at $(X_0,Y_0)=(0,1)$. By the transformation \eqref{eq:xy-uv}, a Brownian motion on the Poincar\'{e} disk $\DDP$ started at the origin with Cartesian coordinates $Z_t \equiv (U_t,V_t) \in \DD$ for $t \in \mathbb{R}_+$ is given by
\begin{equation}\label{eq:DtoH}
U_t = \frac{2X_t}{X_t^2 + (Y_t +1)^2}, ~ \text{and} ~ V_t = \frac{X_t^2 + Y_t^2 -1}{X_t^2 + (Y_t +1)^2}.
\end{equation}
\begin{proof}[Proof of Theorem~\ref{thm:expectation-x-representation}]
Note the formula, immediate from~\eqref{eq:DtoH}, 
\[ 1 - V_t = \frac{2 (Y_t+1)}{X_t^2 + (Y_t +1)^2} , \]
from which it follows that
\begin{equation}
\label{eq:U-V-X}
U_t^2 + V_t^2 - 2V_t +1 = (1-V_t)^2 + U_t^2 = \frac{4}{{X_t^2 + (Y_t +1)^2} }.
\end{equation}
Using the fact that $X_t \to X_\infty$ a.s., and $Y_t \to 0$ a.s., as
in~\eqref{eq:limit-direction}, we get from~\eqref{eq:DtoH} that 
\begin{align}\label{eq:UVinf} U_\infty 
  := \lim_{t \to \infty} U_t = \frac{2 X_\infty}{X_\infty^2 +1} = \cos \theta_\infty, \text{ and }
V_\infty 
 := \lim_{t \to \infty} V_t = \frac{X^2_\infty-1}{X_\infty^2 +1} = \sin \theta_\infty . \end{align}
By the Cauchy formula~\eqref{eq:PeriminD} in $\DDP$, and rotational invariance of $Z_t$, we have that for $t \in \RP$,
\begin{align}\label{eq:Exp-Cauchy}
\EE L_t  = \EE  \perim \conv Z[0,t] 
&= 
\int_0^{2\pi} \EE \left[\sup_{0 \le s \le t} \lambda \left(\varphi,\frac{2Z_s}{1+|Z_s|^2}\right)\right] \ud \varphi
\nonumber\\
&= 
2\pi \EE \left[\sup_{0 \le s \le t} \lambda \left(\pi/2,\frac{2Z_s}{1+|Z_s|^2}\right)\right].
\end{align}
Applying the formula~\eqref{eq:lambda}, the convenient choice of $\phi=\pi/2$ and the transform \eqref{eq:DtoH} yield
\begin{align*}
\lambda \left(\pi/2,\frac{2Z_s}{1+|Z_s|^2}\right)
& =  {\frac{- 2U_s}{1+U_s^2+V_s^2}} \left( {{1 - \frac{2V_s}{1+U_s^2+V_s^2}}} \right)^{-1} \\
& =\frac{-2U_s}{1+U_s^2 + V_s^2 -2V_s} = - X_s \eqd X_s,
\end{align*}
using the formula~\eqref{eq:U-V-X} and distributional symmetry of $X_s$.
In particular, this means that
\[ \sup_{0 \leq s \leq t} \lambda \left(\pi/2,\frac{2Z_s}{1+|Z_s|^2}\right) \eqd X^\star_t . \]
Combined with~\eqref{eq:Exp-Cauchy} this yields $\EE L_t = 2\pi \EE [ X^\star_t]$, as claimed.
\end{proof}

\begin{remark}\label{remark:perim_BM_linesegment}
This approach can also be used to compute $\EE \perim \conv \{ o, B_t \}$, the convex hull of the random line segment, via the Cauchy formula in $\DDP$. It gives that
$$
\EE \perim \conv \{ o, B_t \} = 2\pi \EE \left[\lambda \left(\pi/2,\frac{2Z_t}{1+|Z_t|^2}\right) \vee 0 \right]
$$
since $\lambda(\varphi,0) = 0$. Following the proof above, we find that
$$
\lambda \left(\pi/2,\frac{2Z_t}{1+|Z_t|^2}\right) \vee 0 \eqd X_t^+,
$$
where $x^+ := x \1{ x >0}$, 
so that
\[ \EE \left[\lambda \left(\pi/2,\frac{2Z_t}{1+|Z_t|^2}\right) \vee 0 \right] = \EE \bigl[ X_t^+ \bigr]. \]
For standard Brownian motion $\tW$ on $\RR$, we know $\EE \bigl[ \tW_t^+ \bigr] = \frac{1}{2} \EE \bigl| \tW_t \bigr| = \sqrt{t/(2\pi)}$~\cite[p.~96]{KS}. Since $X_t \eqd \tW_{\xi_t}$ by Lemma \ref{lem:time-change}, we thus have by independence and Lemma \ref{lem:hy-moments}(ii) that $\EE \bigl[ X_t^+ \bigr] = (2 \pi)^{-1/2} \EE \sqrt{ \xi_t} = (2 \pi)^{-1} t +o(t)$, giving
$\EE \perim \conv \{ o, B_t \} = t + o(t)$,
which is one half of $\EE L_t$. As explained in Remark~\ref{rem:radius-line-segment}, from the fact that $\perim \conv \{ o, B_t \} = 2R_t$ it follows that $\EE R_t = t/2 + o(t)$, giving an alternative proof of the asymptotics for $\EE R_t$ in Proposition~\ref{prop:lln}. In our proof of Proposition~\ref{prop:lln} given in \S\ref{sec:lln}, we present a more direct proof using only the SDE~\eqref{eq:r-sde}.
\end{remark}

\section{Radial asymptotics and directional convergence}
\label{sec:lln}

In this section we give a proof of Proposition~\ref{prop:lln} (see Remark~\ref{rem:lln} for comments on which parts of the result are well known). An intermediate step is the following result, which establishes that the radial process never returns to~$0$. In other words, the origin $o$ (and indeed every point) is transient for hyperbolic Brownian motion in $\HH^2$; there are several   ways to establish this fact, but here we use an elementary martingale argument modelled on that for the 2-dimensional Bessel process (see Remark~\ref{rem:0-hit} for the Bessel comparison).

\begin{lemma}
\label{lem:0-hit}
Let $(R_t)_{t \in \RP}$ be an $\RP$-valued, strong Markov 
process with $R_0 \in \RP$ for which the SDE~\eqref{eq:r-sde} is satisfied.
It holds that
$\PP ( R_t > 0 \text{ for all } t >0 ) =1$.
\end{lemma}

\begin{remark}
\label{rem:0-hit}
Write 
$\coth x := 1/\tanh x$. 
It holds for all $x >0$ that 
$\cosh x > 1$, and then some calculus shows that $\sinh x > x$ for all $x >0$ also.
Then the function $f(x) := x \coth x$ has $f'(x) = \frac{\sinh (2x) - 2x}{2 \sinh^2 x} >0$ for all $x>0$, and $f(0+) = 1$, and hence
\begin{equation}
    \label{eq:coth-lower-bound}
  \coth x > \frac{1}{x} , ~\text{for all } x >0. \end{equation}
So Lemma~\ref{lem:0-hit}
can also be established by a comparison with the 2-dimensional Bessel process.
\end{remark}

While for Lemma~\ref{lem:0-hit}  the asymptotics of $\coth x$ near~$0$  are crucial, 
for the large-$t$ behaviour in Proposition~\ref{prop:lln} also important are the asymptotics of $\coth x$ for large $x$. In that direction, we record some preliminary observations. 
 First note that, since  for all $x \in \RR$,~$\coth x - 1 = 2 (\ee^{2x}-1)^{-1}$,
  \begin{equation}
\label{eq:coth-bounds}
1 < \coth x  < 1+ \frac{1}{x} , \text{ for all } x > 0.\end{equation}
 Then from~\eqref{eq:r-sde} we may write
 \begin{equation}
 \label{eq:alpha-def}
     R_t = R_0  + \frac{t}{2} + W_t^R + \alpha_t, \text{ where }
     \alpha_t := \int_0^t \left( \ee^{2 R_s} -1 \right)^{-1} \ud s, \text{ for all } t \in \RP.
 \end{equation}
In particular since $R_t, \alpha_t \geq 0$ for all $t \in \RP$, 
we get from~\eqref{eq:alpha-def} that 
\begin{equation}
    \label{eq:R-lower-bound}
    R_t   \geq   \frac{t}{2} + W_t^R, \text{ for all } t \in \RP.
\end{equation}
From~\eqref{eq:R-lower-bound}, one already gets the lower bounds on $R_t$ and $\EE R_t$ required for~\eqref{eq:R-BM-approx};
the upper bounds need more work, as set out in the proof of Proposition~\ref{prop:lln} below. First, we prove Lemma~\ref{lem:0-hit}.

\begin{proof}[Proof of Lemma~\ref{lem:0-hit}]
Let $(\Ff_t)_{t \in \RP}$
be the filtration to which $R$ is adapted.
For an arbitrary stopping time $T \in [0,\infty]$ 
with respect to the filtration  $(\Ff_t)_{t \in \RP}$, 
and $x \in \RP$, 
define $\lambda_{T,x} := \inf \{ t\in \RP : R_{T \vee t} = x \}$,
the first hitting time of $x$ after time~$T$ if $T < \infty$.  For definiteness,
we interpret $R_\infty := \limsup_{t \to \infty} R_t$ when $T=\infty$. 
 Fix $0 < a < x_0 < b < \infty$,
 and define the event $E := \{  T < \infty, \, R_T = x_0 \}  \in \Ff_T$, and suppose that~$T$ is such that $\PP (E )=1$. (Later in the proof we take $T$ to be deterministic, or the hitting time by $R$ of a positive level, both of which fulfil the hypotheses imposed on $T$.)
Set $\lambda := \lambda_{T,a} \wedge \lambda_{T,b}$,
and let $g(x) := \log x$.
Then from It\^o's formula applied to~\eqref{eq:r-sde}, using the bound~\eqref{eq:coth-lower-bound} with  $g' (x) = 1/x > 0$ for all $x \in (0,\infty)$,
gives, on the event $E$,
\begin{align*}
g( R_{t \wedge \lambda})
- g(R_T) &
=   \int_T^{t \wedge \lambda} R_u^{-1} \ud R_u - \frac{1}{2} \int_T^{t \wedge \lambda}  R_u^{-2} \ud u \\
& \geq  \int_T^{t \wedge \lambda} R_u^{-1} \ud W_u^R =: M_t ,\end{align*}
where $M := (M_t)_{t \in \RP}$ (which depends on $T$ and $\lambda$) 
has quadratic variation $[M] = ([M]_t)_{t \in \RP}$ which 
satisfies
 $\EE ( [ M ]_t  ) \leq t/a^2 < \infty$, so the continuous local martingale $M$ is a genuine martingale~\cite[p.~38]{KS}. Since $\lambda < \infty$, a.s.,
one has $\lim_{t \to \infty} g (R_{t \wedge \lambda} ) = g (R_\lambda)$, and $\sup_{t \in \RP} | g (R_{t \wedge \lambda} ) | \leq \sup_{x \in [a,b]} | \log x | < \infty$. Hence, by the bounded convergence theorem,
$\EE g (R_\lambda) = \lim_{t \to \infty} \EE g (R_{t \wedge \lambda} )
\geq \EE g (R_T)$,
and, since $g (R_T) = g(x_0)$,
\[ \log x_0 = \EE g (R_T) \leq \EE g (R_\lambda) = \PP ( \lambda_{T,a} < \lambda_{T,b} ) \log a + \PP ( \lambda_{T,a} > \lambda_{T,b} ) \log b . \]
Re-arranging, we obtain the gambler's ruin estimate
\[ \PP ( \lambda_{T,a} < \lambda_{T,b} )  \leq \frac{\log (b/x_0)}{\log (b/a)} . \]
In particular, since $ \PP ( \lambda_{T,0} < \lambda_{T,b} ) \leq \PP (\lambda_{T,a} < \lambda_{T,b})$
for every $a >0$ we take $a \downarrow 0$ to obtain $\PP ( \lambda_{T,0} < \lambda_{T,b} ) = 0$
for every $0 < x_0 < b$ and hence $\PP ( \lambda_{T,0} < \infty ) = \lim_{b \to \infty}\PP ( \lambda_{T,0} < \lambda_{T,b} ) =0$ for every $x_0 > 0$.
Thus we have proved that, for every $x_0 > 0$ and every stopping time $T$
such that $\PP ( T < \infty, \, R_T = x_0 ) =1$,  
\begin{equation}
    \label{eq:post-T-escape}
 \PP ( R_t > 0 \text{ for all } t \geq T ) =1 .\end{equation}
 In particular, if $R_0 = x_0 >0$, the case $T = 0$ of~\eqref{eq:post-T-escape}
verifies the claim in the  lemma.

It remains to consider the case where $R_0 = 0$. Now define $\lambda_h := \inf \{ t \in \RP : R_t = h\}$.
Then, since $0 \leq R_{t \wedge \lambda_h} \leq h$,
we obtain from~\eqref{eq:R-lower-bound} that
$h \geq \EE R_{t \wedge \lambda_h}  \geq (1/2) \EE (t \wedge \lambda_h)$,
for every $t \in \RP$, 
from which, by Fatou's lemma, we obtain $\EE \lambda_{1/n} \leq 2/n$ for all $n \in \NN$.
On the event {$\{ \lambda_{1/n} < n^{-1/2} \}$},
the fact~\eqref{eq:post-T-escape} applied at stopping time $T=\lambda_{1/n}$
shows that $R_t > 0$ for all $t \geq n^{-1/2}$, a.s. 
Hence, by continuity along monotone limits and then Markov's inequality,
\begin{align*} \PP ( R_t >0 \text{ for all } t > 0 )
& = \lim_{n \to \infty} \PP (R_t >0 \text{ for all } t \geq n^{-1/2} ) \\
& \geq 1 - \lim_{n \to \infty} \PP ( \lambda_{1/n} \geq n^{-1/2} )
\geq 1 - \lim_{n \to \infty} n^{1/2}\EE \lambda_{1/n} = 1,
\end{align*}
using the fact that $\EE \lambda_{1/n} \leq 2/n$ for all $n \in \NN$. This completes the proof.
\end{proof}

\begin{proof}[Proof of Proposition~\ref{prop:lln}]
From~\eqref{eq:R-lower-bound} we get immediately that $\liminf_{t \to \infty} t^{-1} R_t \geq 1/2$, a.s. 
In particular, $R_t \to \infty$, a.s., which, together with continuity and Lemma~\ref{lem:0-hit} 
shows that
\begin{equation}
\label{eq:R-stay-positive}
   \text{for every } t >0, ~ \inf_{s \geq t} R_s > 0, \as 
\end{equation}
Using the expression~\eqref{eq:alpha-def} for $R_t$ and for $R_1$, we get
\begin{align*}
\sup_{t \geq 1} \, \Bigl| R_t - R_1 - W_t^R - \frac{t}{2} \Bigr| &
\leq \frac{1}{2} + | W_1^R | +  \sup_{t \geq 1} \left| \alpha_t - \alpha_1 \right|.
\end{align*}
Let $\tau := \sup \{ s \in \RP: s + 2W_s^R \leq 2 \}$, which has $\tau < \infty$, a.s.;
then, by~\eqref{eq:R-lower-bound}, $R_t > 1$ for all $t \geq \tau$.
Using the upper bound in~\eqref{eq:coth-bounds} up to time~$1+\tau$, and the
lower bound~\eqref{eq:R-lower-bound} after time~$\tau$,  
we obtain
\begin{align*}
 \sup_{t \geq 1} \left| \alpha_t - \alpha_1 \right| & \leq \frac{1}{2} \int_1^{1+\tau} R_s^{-1}  \ud s  +
 \int_0^\infty \left( \ee^{\max(2 , s + 2 W_s^R)} -1 \right)^{-1} \ud s := \zeta ,\end{align*}
which satisfies $\zeta < \infty$, a.s., using~\eqref{eq:R-stay-positive} together with the facts that $\tau < \infty$ and $\lim_{ s \to \infty} ( s + 2 W_s^R) = \infty$, a.s. This verifies the first statement in~\eqref{eq:R-BM-approx}, from which it 
follows that $\lim_{t \to \infty} R_t/t = 1/2$, a.s.

To prove convergence also of  $\EE R_t/t$, we show that $R_t/t$, $t \geq 1$, is  uniformly integrable, using a Gr\"onwall idea\footnote{We thank Isao Sauzedde for suggesting such an approach.}. 
First, It\^o's formula shows that for $S_t := R_t^2$,
\[ \ud S_t = \left( 1 + R_t \coth R_t \right) \ud t + 2 R_t \ud W^R_t . \]
Hence, using the upper bound in~\eqref{eq:coth-bounds}, we get, for all $t \in \RP$, 
\[ S_t - S_0 \leq 2  t + \int_0^t    \sqrt{S_u}   \ud u + 2 \int_0^t \sqrt{ S_u} \ud W_u^R.   \]
It follows that, for all $t \in \RP$, 
\begin{equation}
\label{eq:gronwall-1}
\EE (S_t ) - \EE (S_0) \leq 2 t + \int_0^t   \EE \sqrt{S_u}   \ud u 
\leq 2t + \int_0^t \sqrt{ \EE (S_u) } \ud u,
\end{equation}
using the Jensen inequality bound $\EE \sqrt{S_u} \leq  \sqrt{ \EE (S_u) }$. 
Write 
\[ g(t) := 2t + \int_0^t  \sqrt{ \EE (S_u) } \ud u, ~\text{and}~ h(t) := \exp \left( \sqrt{ \EE(S_0) + g(t)} \right). \]
Then we can re-write~\eqref{eq:gronwall-1} as
$(g'(t) - 2)^2 = \EE (S_t) \leq \EE (S_0) + g(t)$, and hence, since $g'(t) \geq 2$,
\[ g'(t) \leq 2 + \sqrt{ \EE (S_0) + g(t)} \leq 2 \sqrt{ \EE (S_0) + g(t)},\]
provided $t \geq 2$, since $g(t) \geq 2t$.
Then, by differentiating $h$, we get
\[ \frac{h'(t)}{h(t)} = \frac{g'(t)}{2 \sqrt{ \EE (S_0) + g(t)}} \leq 1, \text{ for all } t \geq 2.
\]
Then, by Gr\"onwall's lemma,
$h(t) \leq h(2) \ee^{t}$ for all $t \geq 2$. By taking logarithms,  we get, for all $t \geq 2$,
\[ \sqrt{g(t)} \leq \sqrt{ \EE (S_0) + g(t) } \leq t + C ,\]
where $C:= \log h(2) < \infty$ is a function of $\EE (S_0)$. Hence $g(t) \leq (t+C)^2$, and, by~\eqref{eq:gronwall-1}, we get that $\EE (S_t) \leq 2t^2$, say, for all $t$ large enough. Consequently $\sup_{t \geq 1} \EE ( (R_t/t)^2 ) < \infty$, meaning that $R_t/t$, $t \geq 1$, is uniformly integrable. Hence  $\lim_{t\to \infty} R_t/t = 1/2$ implies that $\lim_{t\to \infty} \EE R_t/t = 1/2$, also.

We now turn to the convergence of the angular component $\theta$.
To this end, we fix $s >0$ and consider the post-time-$s$ winding process $\Theta^{(s)}$, where we have
 $\Theta^{(s)}_s =0 $ and, 
from~\eqref{eq:theta-sde}, that
$( \Theta^{(s)}_t )_{t \geq s}$
is a local martingale with quadratic variation given by
the increasing process
\begin{equation}
    \label{eq:theta-bracket} [ \Theta^{(s)} ]_t := \int_s^t  \frac{\ud u}{\sinh^2 R_u} , \text{ for } t \geq s. 
\end{equation}
By~\eqref{eq:R-stay-positive},
the integral in~\eqref{eq:theta-bracket} is finite for all $t \geq s$.
Moreover, 
from the fact that $R_t /t \to 1/2$ and~\eqref{eq:theta-bracket}, we deduce that
$[ \Theta^{(s)} ]_\infty := \sup_{t \geq s}   [ \Theta^{(s)} ] _t   < \infty$, a.s.
Since $(\Theta^{(s)}_t  )_{t \geq s}$ is a continuous local martingale, vanishing at~$0$,
the Dambis--Dubins--Schwarz theorem~\cite[pp.~174--5]{KS}  implies that it is a time-change of a one-dimensional Brownian motion $\tW^{(s)}$, say:
 $\Theta^{(s)}_t   = \tW^{(s)}_{[ \Theta^{(s)} ]_t}$, for all $t \geq s$.
Hence we get, for every $s > 0$, 
$\lim_{t \to \infty} \Theta^{(s)}_t =   \tW^{(s)}_{[\Theta^{(s)}]_\infty} =: \Theta^{(s)}_\infty$, a.s.
Since $\theta_t = \theta_s + \Theta^{(s)}_t$, reduced modulo $2\pi$, for all $t \geq s$,
this means that 
$\lim_{t \to \infty} \theta_t = ( \theta_s + \Theta^{(s)}_\infty)$ modulo $2\pi$ exists, a.s.
Moreover, $\lim_{s \to \infty} [\Theta^{(s)}]_\infty = 0$, a.s., so $\lim_{s \to \infty} \Theta^{(s)}_\infty = 0$ in probability. Since also $\theta_s \sim \unif[0,2\pi)$, by taking $s \to \infty$ 
in $\theta_\infty =  ( \theta_s + \Theta^{(s)}_\infty)$ modulo $2\pi$, 
we see
that $\theta_\infty \sim \unif[0,2\pi)$ as well.

Finally, an extension of the preceding argument  quantifies the rate of convergence of $\theta$. 
We give the verification of~\eqref{eq:angular-rate} in the case $s=1$; the general case  works the same way. 
Recall the definition of $\alpha_t$ from~\eqref{eq:alpha-def}. Then
$0 \leq \alpha_t \uparrow \alpha_\infty$ where $\alpha_\infty < \infty$, a.s. Consider the representation 
 $\Theta^{(s)}_t   = \tW^{(s)}_{[ \Theta^{(s)} ]_t}$,  $t \geq s$,
 for the fixed choice $s=1$,
 and, for ease of notation, write simply $\Theta \equiv \Theta^{(1)}$
 and 
  $ \tW \equiv  \tW^{(1)}$ 
  for the rest of this proof. 
 Then define $W'_u := \tW_{[\Theta]_\infty} - \tW_{[\Theta]_\infty-u}$, $u \in \RP$. Here $\tW$ is independent of $[\Theta]$, so $W'_u$ is a Brownian motion
(formally, extend $\tW$ to a two-sided Brownian motion indexed by $\RR$). We are interested, though, in $u \to 0$, since, for all $t \geq 1$, 
\[ \Theta_\infty - \Theta_t = \tW_{[\Theta]_\infty} - \tW_{[\Theta]_t} 
= W'_{\varrho_t}, \text{ where } \varrho_t := {[\Theta]_\infty - [\Theta]_t} =  \int_t^\infty \frac{\ud u}{\sinh^2 R_u} ,
\]
from the $s=1$ case of~\eqref{eq:theta-bracket}.
Since $R_0 =0$, we have from~\eqref{eq:alpha-def} that, as $t \to \infty$,
\[ \sinh R_t = \frac{\ee^{\alpha_\infty}}{2} (1+ o(1)) \ee^{t/2} \ee^{W_t^R} , \as  \]
Writing $W := - W^R$, it follows that
\begin{equation}
\label{eq:rho-tail}
 \varrho_t = 4 (1+o(1)) \ee^{-2\alpha_\infty} \oEe_t, \text{ where } \oEe_t := \int_t^\infty \ee^{2 W_s - s} \ud s .\end{equation}
A crude bound is that, for every $\varepsilon \in (0,1/2)$,
there exists a (random) $t_\varepsilon$ with $t_\varepsilon < \infty$, a.s., such that, for all $t \geq t_\varepsilon$,  
\[  
\oEe_t <  \int_t^\infty \ee^{ - (1-(\varepsilon/2)) s} \ud s \leq \ee^{-(1-\varepsilon) t}  . \]
It follows that, for $t \geq t_\varepsilon$,
\[ \int_t^{2t} \ee^{2W_s-s} \ud s \leq \oEe_t \leq \ee^{-2(1-\varepsilon) t} +  \int_t^{2t} \ee^{2W_s-s} \ud s,
\]
For every $\delta >0$, it holds that, a.s., $\sup_{s \in [t, 2t]}|  W_s | \leq t^{(1/2)+\delta}$ for all $t$ large enough. It follows from these estimates, together with~\eqref{eq:rho-tail}, that we have the $\log$-asymptotic for $\varrho_t$:
\begin{equation}
\label{eq:rho-log-asymptotic}
\lim_{t \to \infty} \frac{ \log (1/\varrho_t )}{t} = 
\lim_{t \to \infty} \frac{ \log (1/\oEe_t) }{t} = 1, \as
\end{equation}
A consequence of 
Khinchin's small-time law of the iterated logarithm~\cite[p.~112]{KS}
is that 
\begin{equation}
\label{eq:khinchin}
\liminf_{u \to 0}  \frac{\log (1/|W'_u|)}{ \log (1/u)} = \frac{1}{2}, \as 
    \end{equation}
Combining~\eqref{eq:rho-log-asymptotic} and~\eqref{eq:khinchin} we obtain
\begin{equation}
    \label{eq:Theta-convergence}
 \limsup_{t \to \infty} \frac{\log | \Theta_\infty - \Theta_t |}{t} = 
 - \liminf_{t \to \infty}  \frac{\log (1/|W'_{\varrho_t}|)}{  t} 
= - \liminf_{t \to \infty}  \frac{\log (1/|W'_{\varrho_t}|)}{ \log (1/\varrho_t)}
=  - \frac{1}{2}, \as \end{equation}
For $t >1$, we can write $\theta_t = \theta_1 + \Theta_t + 2\pi k_t$ for some $k_t \in \ZZ$, and similarly, for $s >1$, $\theta_s = \theta_1 + \Theta_s + 2\pi k_s$ for some $k_s \in \ZZ$. Hence, $\theta_t - \theta_s = \Theta_t - \Theta_s + 2\pi (k_t-k_s)$. As shown above, for every $\varepsilon>0$, there is some $T < \infty$ such that $| \Theta_t-\Theta_s | < \varepsilon$ for all $s, t \geq T$. Since $\theta_t \to \theta_\infty$, a.s., and $\theta_\infty \sim \unif[0,2\pi)$, it also holds that, a.s., $| \theta_t - \theta_s | < \varepsilon$ for all $s, t \geq T'$ for some $T' < \infty$. Hence $k_t - k_s =0$ for all $t,s$ large enough, and so the asymptotic
result for $|\theta_\infty -\theta_t|$ in~\eqref{eq:angular-rate} is deduced from the
corresponding result for $|\Theta_\infty -\Theta_t|$ in~\eqref{eq:Theta-convergence}. 
\end{proof}

\begin{remark}
    \label{rem:uniform-angle}
    Here is an alternative argument to establish that $\theta_\infty \sim \unif [0,2\pi)$, linking to the appearance of the Cauchy limit in Lemma~\ref{lem:cauchy}, rather than using directly the rapid spinning and uniform entrance properties.
    A direct calculation shows that if $\varphi \sim \unif [0,2\pi)$, then $\tan \varphi$ is standard Cauchy;
    moreover, $2 \varphi$ reduced modulo $2\pi$ is also $\unif [0,2\pi)$. Thus from Lemma~\ref{lem:cauchy},
\begin{equation}\label{eq:tanuniform}
X_\infty \eqd \tan ( k \varphi), \text{ for } \varphi \sim \unif [0,2\pi),
\end{equation}
 for every positive integer $k$. 
Thus it follows from \eqref{eq:UVinf} that the limiting angle $\theta_\infty$ satisfies
\[ \tan \theta_\infty = \frac{X_\infty^2 -1}{2 X_\infty} \eqd \frac{\tan^2\varphi -1}{2 \tan \varphi} = -1/\tan (2\varphi) \eqd 1 / X_\infty, \]
by~\eqref{eq:tanuniform} applied in two places. 
Since $X_\infty$ is standard Cauchy, i.e., has density $g(x) = (\pi (1+x^2))^{-1}$ over $x \in \RR$,
a change of variable shows that the density of $\tan \theta_\infty$ is given by
$$
f (x) = x^{-2} g (-1/x) = \frac{1}{\pi (1+x^{-2})}  \cdot \frac{1}{x^2} = \frac{1}{\pi(1+x^2)},
$$
for all $x \in \RR$; i.e.,  $\tan \theta_\infty$ also has a standard Cauchy distribution. Thus, by inverting~\eqref{eq:tanuniform}, we obtain that $\theta_\infty \sim \unif [0,2\pi)$ (recall that, by definition, $\theta_\infty \in [0,2\pi))$.
\end{remark}

\appendix

\section{Exact expectation formulae}
\label{sec:exact}

This short appendix
derives some exact formulae by an appeal to some of the 
impressive exact-distributional results that form one of the main
highlights of the theory of the exponential functionals as discussed in~\S\ref{sec:exponential}. In particular we use results from~\cite{yor92,yor01} (see references therein for prior contributions).
First we give a proof of the
formula from Corollary~\ref{cor:yor} for $\EE L_{T_\lambda}$, the perimeter length at random exponential time~$T_\lambda$.

\begin{proof}[Proof of Corollary~\ref{cor:yor}]
    In Theorem~1 of~\cite[Ch.~6, p.~95]{yor01}, it is shown that
    $\Ee_{T_\lambda}$ has the same distribution as $\beta_{1,a} /(2\gamma_b)$, where in the latter $a, b \in (0,\infty)$ are parameters derived from $\lambda$,     
    $\beta_{1,a}, \gamma_b$ are independent
    random variables, $\beta_{1,a} \sim \textrm{Beta} (1,a)$ has density  $a (1-u)^{a-1}$ on $u \in [0,1]$, and $\gamma_b \sim \textrm{Gamma} (b)$ has density $\Gamma(b)^{-1}\ee^{-s} s^{b-1}$ on $s \in \RP$;
    here $\Gamma$ is the Euler gamma function. The constants $a,b$ are given by $4 a = x - 1$ and $4 b= x + 1$ where $x:=  \sqrt{8 \lambda + 1}$. Then, by independence,
    \[ \EE \sqrt{ \Ee_{T_\lambda} } = \frac{1}{\sqrt{2}} \cdot  \EE \sqrt{ \beta_{1,a} } \cdot \EE \sqrt{1/\gamma_b}, \]
    where, via beta--gamma calculus, and the facts that $y \Gamma (y) = \Gamma (y+1)$ and $\Gamma (1/2) = \sqrt{\pi}$,
    \begin{align*} \EE \sqrt{ \beta_{1,a} } & = a \int_0^1 (1-u)^{a-1} \sqrt{ u }  \ud u = 
    \frac{a \Gamma(a) \Gamma \bigl( \tfrac{3}{2} \bigr)}{\Gamma \bigl( a + \tfrac{3}{2} \bigr)} =    \frac{\sqrt{\pi} \Gamma (a+1)}{2\Gamma \bigl( a + \tfrac{3}{2} \bigr)} , \text{ and }\\
    \EE \sqrt{1/\gamma_b} & = \frac{1}{\Gamma (b)} \int_0^\infty \ee^{-s} s^{b-1}  s^{-1/2} \ud s = \frac{\Gamma \bigl( b - \tfrac{1}{2} \bigr)}{\Gamma (b) } .
    \end{align*}
    From these formulas, the values $a = \tfrac{x-1}{4}$ and $b = \tfrac{x+1}{4}$, and Theorem~\ref{thm:expectation}, we obtain
    \[ \EE L_{T_\lambda} = 
    \pi \cdot\frac{\Gamma \bigl(\tfrac{x+3}{4} \bigr) \Gamma \bigl( \tfrac{x-1}{4} \bigr)}{\Gamma \bigl(\tfrac{x+5}{4} \bigr) \Gamma \bigl( \tfrac{x+1}{4} \bigr)} 
=  \pi  \left( \frac{x-1}{x+1} \right) \left( \frac{\Gamma \bigl( \tfrac{x-1}{4} \bigr)}{ \Gamma \bigl( \tfrac{x+1}{4} \bigr)} \right)^2
    ,\]
    as claimed in~\eqref{eq:cor-yor}. The stated $\lambda \to 0$ asymptotics follow from the facts that $\sqrt{1 + 8 \lambda} -1 \sim 4\lambda$ and $\lambda \Gamma (\lambda) \to 1$ as $\lambda \to 0$ (from above). The $\lambda \to \infty$ asymptotics follow from Stirling's formula: $\Gamma (x) \sim (2\pi x)^{1/2} x^x \ee^{-x}$ as $x \to \infty$.
\end{proof}

Following~\cite{yor92} (among others), we write
\[ A_t^{(\nu)} := \int_0^t \exp (2 W_s + 2 \nu s ) \ud s. \]
Then $A_t^{(-1/2)} = \Ee_t$ as given at~\eqref{eq:xi-clean}.
 As explained in~\cite[p.~510]{yor92}, a Girsanov transformation gives
\[ \EE \bigl[ \bigl( A_t^{(\nu)} \bigr)^{1/2} \bigr] = \ee^{-\nu^2 t/2} \EE \bigl[ \sqrt{A_t} \exp ( \nu W_t ) \bigr] , \]
where $A_t := A^{(0)}_t$. 
Yor~\cite{yor92} gives an expression for the conditional density of $A_t$ given $W_t$,
and deduces a formula for certain mixed moments, which gives
\begin{align*}
    \EE \bigl[ \bigl( A_t^{(\nu)} \bigr)^{1/2} \bigr] & = ( 2 \pi^2 t)^{-1/2} \exp \left( \frac{\pi^2}{2t} \right) \int_0^\infty y^\nu \ud y \int_0^\infty v^{-1/2} \exp \left( - \frac{v}{2} (1+y^2) \right) \psi_{yv} (t) \ud v, 
\end{align*}
(obtained from equation (6.e) of~\cite[p.~528]{yor92} with $f(y) = y^\nu$ and $g(z) = \sqrt{z}$), where
\begin{equation}
    \label{eq:psi-def}
 \psi_u (t) :=  \int_0^\infty \exp \left( - \frac{z^2}{2 t} \right) \exp \left( -u \cosh z \right) \sin \left( \frac{\pi z}{t} \right) \sinh z \ud z. \end{equation}
So taking $\nu = -1/2$ and using Theorem~\ref{thm:expectation} gives the multiple-integral formula
\begin{align} 
\label{eq:exact-formula-yor}
  \EE L_t = \sqrt{8 \pi} \EE \sqrt{ \Ee_t }   & = \frac{2}{\sqrt{\pi t}}  \exp \left( \frac{\pi^2}{2t} \right) \int_0^\infty  \frac{\ud y}{\sqrt{y}} \int_0^\infty   \exp \left( - \frac{v}{2} (1+y^2) \right) \psi_{yv} (t) \frac{\ud v}{\sqrt{v}}, 
\end{align}
where~$\psi$ is defined at~\eqref{eq:psi-def}.

\section*{Acknowledgements}
\addcontentsline{toc}{section}{Acknowledgements}

 CB was supported in part by the German Research Foundation (DFG) Project 531540467.
 AW is supported by EPSRC grant EP/W00657X/1.   Some of this work was carried
 out during the programme ``Stochastic systems for anomalous diffusion'' (July--December 2024) hosted by the  Isaac Newton Institute, under EPSRC grant EP/Z000580/1. The authors thank Isao Sauzedde   and two anonymous referees for their comments and helpful suggestions.

 \bibliographystyle{abbrvnat} 
 \setlength{\bibsep}{4pt plus 0.4ex}
\bibliography{BVW_hyperbolic_hulls.bib}

\end{document}